\documentclass[11pt]{amsproc}
\usepackage{graphicx, amssymb,amsmath, latexsym,cite,fancyhdr}
\usepackage{amsthm,arydshln}
\usepackage{times}
\usepackage{float}
\usepackage{tikz}
\usetikzlibrary{arrows,shapes,positioning,calc}
\usepackage[T1]{fontenc}
\usepackage[ruled,linesnumbered,norelsize]{algorithm2e}
\usepackage{url}
\usepackage{colortbl}
\usepackage{longtable}
\usepackage[all]{xy}

\newtheoremstyle{mythm}
{6pt}
{6pt}
{\it}
{}
{\bf}
{.}
{.5em}
{}

\newtheoremstyle{mydef}
{6pt}
{6pt}
{}
{}
{\bf}
{.}
{.5em}
{}

\newtheoremstyle{myrem}
{6pt}
{6pt}
{}
{}
{\bf}
{.}
{.5em}
{}

\newtheoremstyle{myex}
{6pt}
{3ex}
{}
{}
{\bf}
{.}
{.5em}
{}

\theoremstyle{mythm}
\newtheorem{theorem}{Theorem}
\newtheorem{lemma}[theorem]{Lemma}
\newtheorem{proposition}[theorem]{Proposition}

\theoremstyle{mydef}
\newtheorem{definition}[theorem]{Definition}

\theoremstyle{myrem}
\newtheorem{remark}[theorem]{Remark}
\numberwithin{equation}{section}

\theoremstyle{myex}
\newtheorem{example}[theorem]{Example}

\newcommand{\ad}{\mathop{\rm ad}}

\newcommand{\Span}{{\rm Span}}
\renewcommand{\sl}{{\mathfrak{sl}}}
\newcommand{\gl}{{\mathfrak{gl}}}
\newcommand{\GL}{{\rm GL}}
\newcommand{\SL}{{\rm SL}}

\newcommand{\sign}{{\rm sgn}}
\newcommand{\diag}{\mathrm{diag}}
\newcommand{\g}{{\mathfrak g}}
\newcommand{\fa}{{\mathfrak a}}
\newcommand{\fb}{{\mathfrak b}}

\newcommand{\fn}{{\mathfrak n}}

\newcommand{\fk}{{\mathfrak k}}
\newcommand{\fp}{{\mathfrak p}}

\newcommand{\End}{{\rm End}}

\newcommand{\fs}{{\mathfrak s}}
\newcommand{\ft}{{\mathfrak t}}

\newcommand{\fz}{{\mathfrak z}}
\newcommand{\h}{{\mathfrak h}}

\newcommand{\exdone}{\hfill$\bullet$}

\newcommand{\C}{\mathbb{C}}

\newcommand{\Ad}{\mathrm{Ad}}
\newcommand{\Int}{\mathrm{Int}}

\newcommand{\R}{\mathbb{R}}
\newcommand{\Z}{\mathbb{Z}}
\newcommand{\Aut}{{\rm Aut}}

\newcommand{\B}{\mathcal{B}}
\newcommand{\G}{\mathcal{G}}
\newcommand{\K}{\mathcal{K}}
\newcommand{\mc}[1]{\mathbf{#1}}

\newcounter{ithmcount}

\newenvironment{items}{
\begin{list}{$\alph{item})$}
{\labelwidth30pt \leftmargin30pt \topsep3pt \itemsep2pt \parsep0pt}}
{\end{list}}

\allowdisplaybreaks

\begin{document}


\title{Nilpotent orbits in real symmetric pairs and stationary black holes}

\author[H.\ Dietrich]{Heiko Dietrich}
 \address{School of Mathematical Sciences, Monash University, VIC 3800, Australia}
 \email{heiko.dietrich@monash.edu}
 \author[W.\ A.\ de Graaf]{Willem A.\ de Graaf}
 \address{Department of Mathematics, University of Trento, Povo (Trento), Italy}
\email{degraaf@science.unitn.it}
\author[D.\ Ruggeri]{Daniele Ruggeri}
\address{Universit\`a di Torino, Dipartimento di Fisica
and I.N.F.N. - sezione di Torino, Via P. Giuria 1, I-10125 Torino, Italy}
\email{daniele.rug@gmail.com}
\author[M.\ Trigiante]{Mario Trigiante}
\address{DISAT, Politecnico di Torino, Corso Duca degli Abruzzi 24, I-10129
Torino, Italy}
\email{mario.trigiante@polito.it}

\begin{abstract}
In the study of stationary solutions in extended supergravities with symmetric scalar manifolds, the nilpotent orbits of a real symmetric pair play an important role. In this paper we discuss two approaches to determine the nilpotent orbits of a real
symmetric pair. We apply our methods to an explicit example, and thereby
classify the nilpotent orbits of $(\SL_2(\R))^4$ acting on the fourth tensor
power of the natural 2-dimensional  $\SL_2(\R)$-module.
This makes it possible to classify all stationary solutions of
the so-called STU-supergravity model.
\end{abstract}

\thanks{Dietrich was supported by an ARC DECRA (Australia), project DE140100088.}


\maketitle

\reversemarginpar

\section{Introduction}
\noindent Studying and classifying the nilpotent orbits of a (real or complex) semisimple Lie group has drawn a lot of attention in the mathematical literature, we refer to the book of Collingwood \& McGovern \cite{colmcgov} or the recent papers \cite{dfg2,gra15,litt8} for more details and references.  Besides their intrinsic mathematical importance, nilpotent orbits also have a significant bearing on theoretical physics, in particular, on the problem of studying (multi-center) asymptotically flat black hole solutions to extended supergravities, see for example \cite{rev1,rev2,rev3,rev4,BPSmc,Goldstein:2008fq,Bena:2009ev,bossard21,bossard22}. Of particular relevance in that context are real symmetric pairs $(\g,\g_0)$, that is, real semisimple Lie algebras $\g$ which admit a $\Z/2\Z$-grading $\g=\g_0\oplus\g_1$. In ungauged 4-dimensional supergravity models featuring a symmetric scalar manifold, all stationary solutions (which are locally asymptotically flat) admit an effective description as solutions to a 3-dimensional sigma-model with symmetric, pseudo-Riemannian target space. In particular they fall within orbits of the isotropy group $G_0$ of this symmetric target space, which is a real semisimple non-compact Lie group, acting on the tangent space $\g_1$ to which the Noether charge matrix of the solution belongs. If the black hole solution is extremal, namely has vanishing Hawking temperature, then the corresponding Noether charge matrix is nilpotent and thus belongs to a nilpotent $G_0$-orbit on $\g_1$, see for example \cite{Gaiotto:2007ag,Bergshoeff:2008be,Bossard:2009at}. We recall that a $G_0$-orbit is nilpotent if its closure contains~0; this is the reason why such  orbits are also called unstable.
So far the classification of such solutions was mainly based on the \emph{complex} nilpotent orbits of the complexification $G_0^c$ acting on $\g_1^c$, see for example \cite{bossard21,bossard22}. By the Kostant-Sekiguchi bijection, these complex orbits  are in one-to-one correspondence to the real nilpotent orbits of $G$ acting on its Lie algebra $\mathfrak{g}$.\footnote{See \cite{deBoer:2014iba,Deger:2015tra} for recent applications of this classification to the study of supersymmetric string solutions.}  On the other hand, \emph{real} nilpotent orbits of $G_0$ acting on $\g_1$ provide a more intrinsic characterization of regular single-center solutions (that is, black hole solutions which do not feature curvature singularities): each $G_0^c$-orbit accommodates in general singular as well as regular solutions, which can be distinguished by their $G_0$-orbits. The notion of $G_0$-orbits also provide stringent, $G$-invariant regularity constraints on multi-center solutions: A necessary regularity condition for a multi-black hole system to be regular is that each of its constituents is regular \cite{bossard21} and this in turn translates into a condition on their $G_0$-orbits.

In this paper we illustrate the importance of real nilpotent orbits by considering single-center solutions to a simple 4-dimensional model, namely the so-called STU model, see for instance  \cite{Bergshoeff:2008be,bossard21}. We briefly provide the physical motivation for this problem (-- referring to \cite{mariodaniele} for a more detailed discussion of multi-center solutions --) and then attack it using a purely mathematical approach.
 More generally, we describe the mathematical framework for two methods which can be used to list the nilpotent orbits of a Lie group that has been constructed from a real symmetric pair.

\subsection{Results and structure of the paper}
In Section \ref{secbackground}, we give details on the physical background and motivation of this paper. In Section \ref{secsetup}, our  mathematical set-up is outlined and relevant definitions are given. For greater generality, to each real symmetric pair $(\g,\g_0)$ with corresponding grading $\g=\g_0\oplus\g_1$ we associate a class of Lie groups (rather than just one group) acting on $\g_1$; we show that each of these groups is reductive (in the sense of Knapp \cite{knapp}). In Section \ref{secbackground} we also formally describe the main example considered in this paper: it is constructed from a $\Z/2\Z$-grading of a real Lie algebra $\g$ of type $D_4$, and leads to a representation of the Lie group $G_0=(\SL_2(\R))^4$ on the space $\g_1\cong V_2\otimes V_2\otimes V_2\otimes V_2$, where $V_2$ is the natural 2-dimensional $\SL_2(\R)$-module. In \cite{bddmr} this this representation has been considered over the complex field, and it is shown that there are 30 nonzero  nilpotent orbits. The methods we develop here will be used to show that there are 145 nonzero nilpotent orbits over the real numbers; 101 of these orbits are relevant to the study of the STU-model solutions introduced in Section \ref{secbackground}. Figure \ref{Figfigs} summarises the results.
\begin{figure}[ht]
\centering
\includegraphics[width=17cm]{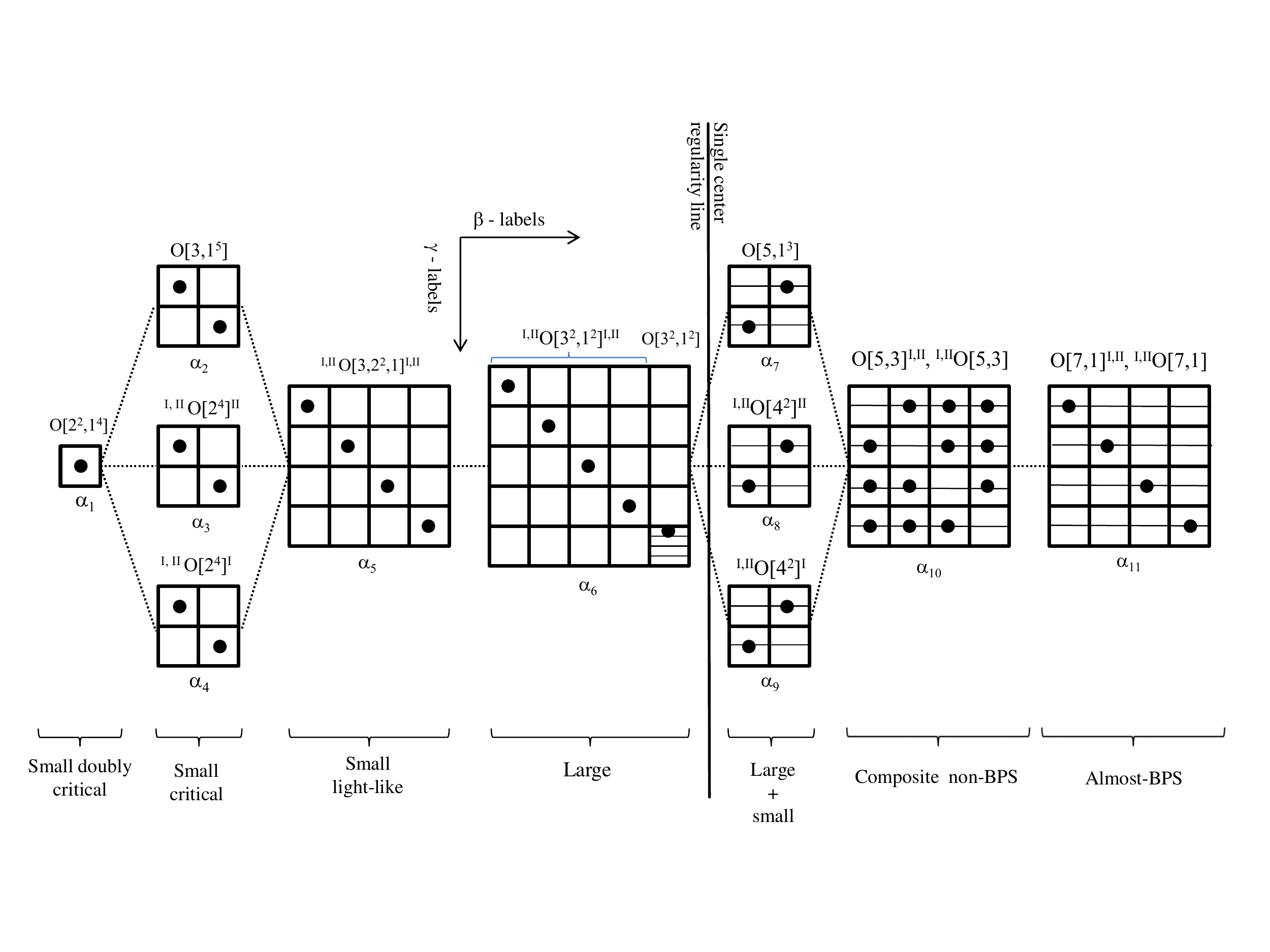}
 \caption{\small Schematic representations of the nilpotent orbits of $G_0={\rm SL}_2(\mathbb{R})^4$ on the coset space $\mathfrak{g}_1$ of ${\rm SO}(4,4)/{\rm SL}_2(\mathbb{R})^4$. Each square block represents an ${\rm SO}(4,4)^\mathbb{C}$-nilpotent orbit in its Lie algebra, while each column is in one-to-one correspondence with ${\rm SO}_0(4,4)$-nilpotent orbits in $\mathfrak{so}(4,4)$. For their description we use the notation of \cite{colmcgov,Bossard:2009we,bddmr} (the trivial orbit $[1^8]$ is omitted). Thick vertical and horizontal lines separate orbits with distinct $\beta$- and $\gamma$-labels, respectively.
 The orbits $\alpha^{(2)},\alpha^{(3)}, \alpha^{(4)}$ as well as the orbits $\alpha^{(7)},\alpha^{(8)}, \alpha^{(9)}$ are related by STU triality, which is the outer-automorphism of the $D_4$ algebra $\mathfrak{g}^c$. The empty slots do not contain regular solutions \cite{mariodaniele}.  The orbit structure with respect to the ${\rm SO}(2,2)^2$ subgroup of ${\rm SO}(4,4)$, which is relevant to the study of stationary solutions, is obtained by removing the thin horizontal lines in the $\alpha^{(7)},\dots, \alpha^{(11)}$-blocks, thus halving the corresponding number of cells and yielding a total of 101 orbits.}\label{Figfigs}
\end{figure}

In Sections \ref{triples} and  \ref{carrier1}, as our first main result, we describe two methods for listing the
nilpotent orbits of a real symmetric pair. Variations of these methods have been used in the literature, and to both we make useful additions. The general outline of these methods is the same: first one determines a finite set of nilpotent
elements that contains representatives of all orbits; second, these elements are shown to be non-conjugate by using a variety of arguments. More precisely, in Section \ref{triples}  we review the classification procedure of real
nilpotent orbits used  in  \cite{Fre:2011uy}; this procedure is based on finding certain special $\mathfrak{sl}_2$-triples and uses  tensor classifiers. However, it has not been shown  in \cite{Fre:2011uy} that one can always find such triples. Here we rigorously prove that. For the implementation of this method
we use the system {\sc Mathematica} \cite{mathematica}, mainly because of its
equation solving abilities.
In Section \ref{carrier1}, we summarise the classification procedure based on
Vinberg's theory of carrier algebras \cite{vinberg2}, which was extended to
the real case in \cite{dfg2}. In that paper some confusing assumptions
have been posed on the Lie group that is used;  we clarify this here. For the
implementation of this method we use the computational algebra system
{\sf GAP}4 \cite{GAP4} and its package {\sf CoReLG} \cite{corelg}.

We remark that V{\^a}n L{\^e} \cite{vanle} has developed a third strategy  for listing the nilpotent orbits
of a symmetric pair; however, her  method requires to solve a difficult problem in algebraic geometry and therefore, to the best of our knowledge, has yet not led to a practical algorithm or implementation.

In Section \ref{SNO}, as our second main result, we introduce  some  mathematical invariants and methods for distinguishing real nilpotent orbits. We discuss the so-called $\alpha$-, $\beta$-, and $\gamma$-labels, tensor classifiers,
and a (rather brute force) method based on solving polynomial equations using the technique of Gr\"obner bases.

We apply the two approaches described in Sections \ref{triples} and \ref{carrier1} to our main example of the STU model, and obtain  the same classification. This classification is our third main result, and we report on our findings in Section \ref{sectab}.
Here we note that both
methods have their advantages and drawbacks: An advantage of
the method based on $\sl_2$-triples is that it produces so-called Cayley triples, which gives a straightforward algorithm to
compute the $\beta$-label of the orbit. An advantage of the method based on carrier algebras is that it produces
representatives with ``nice'' coefficients: in our main example, these coefficients are $\pm1$, see the orbit representatives in
Table \ref{tabletot}. Clearly, having two methods also allows for a convenient cross-validation of our classification results.

\section{Background and physical motivation}\label{secbackground}

\noindent One of the physical motivations behind the study of nilpotent orbits of real semisimple Lie groups is the problem of studying asymptotically-flat black hole solutions to extended (that is, $\mathcal{N}>1$) ungauged supergravities.\footnote{Here we restrict attention  to supergravity models with symmetric homogeneous scalar manifolds.} These theories feature characteristic global symmetry groups of the field equations, and Bianchi identities. Such groups act on the scalar fields as isometry groups of the corresponding scalar manifold, and, at the same time, through generalised electric-magnetic transformations on the vector field strengths and their magnetic duals, see \cite{Gaillard:1981rj}. In \cite{Breitenlohner:1987dg} it was found that a subset of all solutions to the 4-dimensional theory, namely, the stationary (locally-)asymptotically-flat ones \cite{rev1,rev2,rev3,rev4}, actually feature a larger symmetry group $G$ which is not manifest in four space-time dimensions ($D=4$),  but rather in an effective Euclidean 3-dimensional description which is formally obtained by compactifying the 4-dimensional model along the time direction and dualising the vector fields into scalars. Stationary 4-dimensional asymptotically-flat black hole solutions can be conveniently arranged in orbits with respect to this larger symmetry group $G$, whose action  has proven to be a valuable tool for their classification (see \cite{Cvetic:1995kv,Gunaydin:2005mx,Gaiotto:2007ag,Bergshoeff:2008be,Bossard:2009at,Chemissany:2009hq,Bossard:2009we,Kim:2010bf,Chemissany:2010zp,Fre:2011uy,bossard21,bossard22,Chemissany:2012nb}). It also yields a ``solution-generating technique''  (see \cite{Cvetic:1995kv}) for constructing new solutions from known ones (see \cite{Cvetic:2013vqi,Andrianopoli:2013kya,Andrianopoli:2013jra,Chow:2013tia,Chow:2014cca}).

\subsection{Asymptotically flat black holes and nilpotent orbits}
In the effective $D=3$ description, stationary asymptotically-flat 4-dimensional black holes are solutions to an Euclidean non-linear sigma-model coupled to gravity, the target space being a pseudo-Riemannian manifold $\mathcal{M}$ of which $G$ is the isometry group. Such solutions are described by a set of scalar fields $\phi^I(x^i)$ parametrising $\mathcal{M}$, which are functions of the three spatial coordinates $x^1,x^2,x^3$; in the axisymmetric solutions the dependence is restricted to the polar coordinates $(r,\theta)$ only. The asymptotic data defining the solution comprise the value $\phi_0\equiv (\phi^I_0)$ of the scalar fields at radial infinity and the Noether  charge matrix $Q$, which is  associated with the global symmetry group of the sigma-model and which has value in the Lie algebra $\mathfrak{g}$ of $G$. If $\mathcal{M}$ is homogeneous, then we can always fix $G$ to map the point at infinity $\phi_0$ into the origin $O$, where the invariance under the isotropy group $G_0$ of $\mathcal{M}$ is manifest.\footnote{In contrast to  \cite{Chemissany:2012nb}, to uniform our notation with mathematical convention, here we denote the isotropy group of $\mathcal{M}$ by $G_0$ and its Lie algebra by $\g_0$,  instead of $H^*$ and $\mathfrak{H}^*$; moreover, we denote the coset space by $\mathfrak{g}_1$ instead of $\mathfrak{K}^*$.} We restrict ourselves only to models in which $\mathcal{M}$ is homogeneous symmetric of the form $\mathcal{M}=G/G_0$. The solutions are therefore classified according to the action of $G_0$ (residual symmetry at the origin) on the Noether charge matrix $Q$, seen as an element of the tangent space to the manifold in $O$. The rotation of the solution is encoded in another $\mathfrak{g}$-valued matrix $Q_\psi$, first introduced in \cite{Andrianopoli:2012ee,Andrianopoli:2013kya}, which contains the angular momentum of the solution as a characteristic component and vanishes in the static limit. Once we fix $\phi_0\equiv O$, both $Q$ and $Q_\psi$ become elements of the coset space $\g_1$ (which is isomorphic to the tangent space at the origin) and thus transform under $G_0$. The action of $G$ on the whole solution amounts to the action of $G_0$ on $Q$ and $Q_\psi$.

 Non-extremal (or extremal over-rotating) solutions are characterized by matrices  $Q$ and $Q_\psi$  belonging to the same regular $G_0$-orbit which contains the Kerr (or the extremal-Kerr) solution.  In the so-called  STU model, which is an $\mathcal{N}=2$ supergravity coupled to three vector multiplets, the most general representative of the Kerr-orbit was derived in \cite{Chow:2013tia,Chow:2014cca} and features all the duality-invariant properties of the most general solution to the maximal (ungauged) supergravity of which the STU model is a consistent truncation. (The name of this model comes from the conventional notation $S$, $T$, and $U$ for  the three complex scalar fields in these multiplets). On the other hand, extremal static and \emph{under-rotating} solutions \cite{Rasheed:1995zv,Larsen:1999pp,Astefanesei:2006dd} feature nilpotent $Q$ and $Q_\psi$ which belong to \emph{different} orbits of $G_0$. The classification of these solutions is therefore intimately related to the classification of the nilpotent orbits in a  given representation $\rho$ of a real non-compact semisimple Lie group -- which is the general mathematical problem we focus on in this paper: here the representation $\rho$ is defined by the adjoint action of $G_0$ on the coset space $\mathfrak{g}_1$ which $Q$ and $Q_\psi$ belong to, once we fix $\phi_0\equiv O$.

Stationary extremal solutions have been studied in \cite{Bossard:2009we,bossard21,bossard22} in terms of the nilpotent orbits of the complexification $G^c_0$ of $G_0$; the latter are known from the mathematical literature. As far as single-center solutions are concerned, as mentioned in the introduction, these orbits, as opposed to the real ones, do not provide an intrinsic characterisation of regular single-center solutions, since in general they contain  singular solutions as well as regular ones. A classification of real nilpotent orbits has been performed in specific $\mathcal{N}=2$ ungauged models \cite{Fre:2011ns,Fre:2011uy,Chemissany:2012nb}, in connection to the study of their extremal  4-dimensional solutions. There it is shown that, at least for single-center black holes, there is a one-to-one correspondence between the regularity of the solutions\footnote{Here, somewhat improperly, we use the term regular also for \emph{small black holes}, namely solutions with vanishing horizon area; these are limiting cases of regular solutions with finite horizon-area, which are named \emph{large black holes}.} (as well as their supersymmetry) and certain real nilpotent orbits. This allows us to check the regularity of the solution by simply inspecting the corresponding $G_0$-orbit. The classification procedure adopted in \cite{Fre:2011ns,Fre:2011uy,Chemissany:2012nb} combines the method of standard triples \cite{colmcgov}
with new techniques based on the Weyl group: After a general group
theoretical analysis of the model, this approach allows for a
systematic construction of the various nilpotent orbits by solving
suitable matrix equations in nilpotent generators. Solutions to
these equations belong to the same $G_0^c$-orbit, but in general to different $G_0$-orbits. The final step is to group the solutions under the action of $G_0$. Solutions which are not in the same $G_0$-orbit are  distinguished by certain $G_0$-invariants, amongst others, \emph{tensor classifiers}, that is, signatures of suitable $G_0$-covariant symmetric tensors. This ensures that the classification is complete.

The main difficulty in determining the nilpotent $G_0$-orbits in $\mathfrak{g}_1$ is that such orbits are not completely classified by the intersection of the $G_0^c$-orbits in $\mathfrak{g}_1^c$ and the $G$-orbits in its Lie algebra $\mathfrak{g}$, both of which are known: The former are completely classified by the so-called $\gamma$-labels; the latter by the so-called $\beta$-labels obtained by the Kostant-Sekiguchi Theorem.
These two labels do not provide a complete classification of the real nilpotent orbits, as it was shown in an explicit example in \cite{Chemissany:2012nb}.  Distinct $G_0$-orbits having the same $\gamma$- and the same $\beta$-labels can be characterised using $G_0$-invariant quantities, like tensor classifiers, and will be distinguished by a further label $\delta$.  We refer to Section \ref{tclass} for a precise definition of all the aforementioned labels. 

\subsection{Our main example}\label{secourex}
Of particular interest are the multi-center solutions like the \emph{almost-BPS} ones \cite{Goldstein:2008fq,Bena:2009ev,Dall'Agata:2010dy,Galli:2010mg} and the \emph{composite non-BPS} ones \cite{bossard21}. These are extremal solutions (with zero Hawking temperature), characterised by nilpotent matrices $Q$ and $Q_\psi$. Since the geometries of the horizons surrounding each center are affected by subleading corrections (due to the interaction with the other centers), the regularity of the whole solution implies that each center, if isolated from the others, is regular \cite{bossard21}. Real nilpotent orbits (namely, the $G_0$-orbits in $\mathfrak{g}_1$), as opposed to the complex ones, provide an intrinsic characterisation of regular single center solutions, and thus are a valuable tool for constructing regular multi-center solutions: A necessary condition for a multi-center solution to be regular is that the Noether charge of each of its centers belongs to real orbits which correspond to regular single-center solutions, their sum coinciding with the total Noether charge.\footnote{In \cite{mariodaniele} this statement is made more precise by defining an \emph{intrinsic} $G_0$-orbit for each center, since, strictly speaking, the Noether charges of each constituent black hole do not belong to $G_0$-orbits. This is done by associating with each center an \emph{intrinsic Noether charge matrix} referring to the non-interacting configuration where the distances between the centers are sent to infinity.}  A detailed discussion of this matter  (in relation to multi-center solutions) is given in \cite{mariodaniele}; the aim of the present paper is to  illustrate the importance of considering real nilpotent orbits: we show that a complex orbit contains solutions which, although exhibiting an acceptable behaviour of the metric close to the center and at infinity, feature singularities at finite distances from the center.

For this purpose we consider solutions to the simple $D=4$ STU model, see for instance \cite{Bergshoeff:2008be,bossard21}.  The corresponding effective $D=3$ description of stationary solutions has $G={\rm SO}(4,4)$ as global symmetry group and the scalar fields span the manifold $\mathcal{M}_{\rm scal}=G/G_0'$ with isotropy group \[G_0'=({\rm SO}(2,2))^2={\rm SO}(2,2)\times{\rm SO}(2,2),\]
where  ${\rm SO}(2,2)={\rm SL}_2(\mathbb{R})\times_{\mathbb{Z}_2}{\rm SL}_2(\mathbb{R})$ is a central product of two ${\rm SL}_2(\mathbb{R})$. The Lie group $G_0'$ is locally
isomorphic to \[G_0=({\rm SL}_2(\mathbb{R}))^4\] as it has the same Lie algebra as $G_0$.
The extremal solutions, once we fix $\phi_0\equiv O$, are characterised by a Noether charge matrix $Q$ in some nilpotent orbit of $G_0'$ over the coset space $\mathfrak{g}_1$. This is the example we explicitly work out in the present paper, with the difference that we actually consider $G_0$-orbits rather than $G_0'$-orbits. In Table \ref{tabletot}, we list  $G_0$-orbit representatives and associated $\alpha$-, $\beta$-, and $\gamma$-labels. The classification of the orbits with respect to $G_0$ only differs from that corresponding to $G_0'$ by a simple identification. This identification is described by the following rule,
for whose explanation we refer to \cite{mariodaniele}.
In Table \ref{tabletot},  every pair of $G_0$-orbits which have the same $\alpha$-label in $\{\alpha^{(7)},\ldots,\alpha^{(11)}\}$, coinciding $\gamma$- and coinciding $\beta$-labels, and which are otherwise only distinguished by  $\delta^{(1)}$ and $\delta^{(2)}$, define the same $G_0'$-orbit:  for example, the two $G_0$-orbits with labels $\alpha^{(10)}\beta^{(10;1)}\gamma^{(10;2)}\delta^{(1)}$ and  $\alpha^{(10)}\beta^{(10;1)}\gamma^{(10;2)}\delta^{(2)}$ define the same $G_0'$-orbit. We obtain $145$ orbits under $G_0$, and these  reduce to $101$ orbits under $G_0'$.

 For regular or small (that is, with vanishing horizon area) single-center extremal solutions, $Q$ (in the fundamental representation of $G$) must have a degree of nilpotency not exceeding 3. This restricts the $\alpha$-label to be in the set $\{\alpha^{(1)},\dots,\alpha^{(6)}\}$. More specifically, they are characterised by coinciding $\gamma$- and $\beta$- labels \cite{Kim:2010bf,Chemissany:2012nb}. Orbits with $\alpha$-label in the set $\{\alpha^{(7)},\dots,\alpha^{(11)}\}$ can only describe regular multi-center solutions \cite{bossard21}. Regular single-center solutions correspond to orbits with $\alpha$-label $\alpha^{(6)}$. As far as the static solutions are concerned, we have three types:


\hspace*{-0.5cm}\begin{tabular}{rll}
i)& $\alpha^{(6)}$, $\gamma^{(6;1)}$, $\beta^{(6;1)}$ &(BPS solutions)\\
ii)&$\alpha^{(6)}$, $\gamma^{(6;2,3,4)}$, $\beta^{(6;2,3,4)}$ & (non-BPS solutions with vanishing central charge at the horizon)\\
iii)& $\alpha^{(6)}$, $\gamma^{(6;5)}$, $\beta^{(6;5)}$, $\delta^{(1)}$& (non-BPS solutions with non-vanishing central charge at the horizon).
\end{tabular}

\vspace*{1ex}

\noindent These orbits correspond to the classification of 4-dimensional regular static solutions obtained  in \cite{Bellucci:2006xz}. The complex nilpotent orbits are only characterised by the $\alpha$- and $\gamma$-labels, and thus comprise real orbits with different $\beta$-labels --  some of which describe solutions featuring singularities at finite radial distance from the centers. To illustrate this, it is useful to describe the single-center generic representative of the orbits with $\alpha$-label ranging from $\alpha^{(1)}$ to $\alpha^{(6)}$ in terms of the \emph{generating solution} \cite{Bergshoeff:2008be,Chemissany:2009hq,Chemissany:2012nb} (that is, the representative of the real orbits which depends on the least number of parameters). The space-time metric is expressed in terms of four harmonic functions
\begin{equation}
{\bf H}_0=1-k_0\,\tau=1-\sqrt{2}\,\epsilon_0\,q_0\,\tau\quad\text{and}\quad {\bf H}_\ell=1-k_\ell\,\tau=1-\sqrt{2}\,\epsilon_\ell\,p^\ell\,\tau \quad (\ell=1,2,3),
\end{equation}
where $\epsilon_\ell=\pm 1$ and $\ell=0,1,2,3$ and  $\tau\equiv -1/r$ (with $r$ being the radial distance from the center), and $q_0,p^1,p^2,p^3$ are the electric ($q_0$) and magnetic ($p^\ell$) charges of the solution. The metric of the static solution reads
\begin{align}
{\rm d}s^2&=-e^{2U}\,{\rm d}t^2+e^{-2U}\,({\rm d}r^2+r^2{\rm d}\theta^2+r^2\sin^2(\theta){\rm d}\varphi^2) \quad ,\nonumber\\
e^{-2U}&=\sqrt{{\bf H}_0 {\bf H}_1 {\bf H}_2 {\bf H}_3}\qquad\text{and}\qquad e^{\varphi_i}=\sqrt{({\bf H}_0 {\bf H}_i)/({\bf H}_j {\bf H}_k)}\text{ with } \{i,j,k\}=\{1,2,3\},\quad\label{gsol2}
\end{align}
where $e^{\varphi_i}$ are the imaginary parts of the three complex scalars, the real parts being zero on the solution.
Asymptotic flatness requires $\lim_{r\rightarrow \infty} U(r)=0$.
As shown in \cite{Chemissany:2012nb}, the $\gamma$-label of the orbit only depends on $\epsilon_\ell k_\ell^2$ while the $\beta$-label depends on $\epsilon_\ell k_\ell$, $\ell=0,1,2,3$; they coincide only for $k_\ell>0$, which is the necessary and sufficient condition for regularity. Indeed, if one of the $k_\ell$ were negative, then some of the harmonic functions would have a zero root, and the metric a singularity at finite $r$ (or, equivalently, $\tau$). One can show that this value of $r$ corresponds to a curvature singularity. The area of the horizon is given by
\begin{equation}
A_{\rm H}=4\pi\lim_{\tau\rightarrow-\infty} e^{-2U}/\tau^2=4\pi\sqrt{k_0 k_1 k_2 k_3},\label{AU}
\end{equation}
while the ADM mass reads
\begin{equation}
M_{{\rm ADM}}=\lim_{\tau\rightarrow 0^-}\dot{U}=(k_0+k_1+k_2+k_3)/4\,.\label{ADM}
\end{equation}
We see that if only two of the $k_\ell$ are negative, then the solution can be singular, but with acceptable near-horizon limit (according to (\ref{AU})) and positive ADM mass (\ref{ADM}). This is illustrated in Figure \ref{fig1}, where we consider (single-center) representatives of real orbits within the same complex one. It is shown that within the complex orbit of the  BPS solutions (defined by  $\gamma^{(6;1)}$) and of the non-BPS solutions of type iii) (defined by  $\gamma^{(6;5)}$), one can find solutions (Figures \ref{fig1}a) and \ref{fig1}c)) which cannot be distinguished from the regular ones by the asymptotic behaviour of their metric
at $r\rightarrow 0$ and $r\rightarrow \infty$, but which feature singularities at finite $r$. Such solutions are distinguished from the regular ones by their real nilpotent orbits. Therefore the framework of real nilpotent orbits is the appropriate one to characterise, in an intrinsic algebraic way, the regularity of single and  multi-center solutions to ungauged supergravity. An equivalent approach is to implement regularity directly on the solution, as it is done in \cite{Bossard:2013nwa} where a detailed analysis is made of the composite non-BPS solutions and a characterization of the regularity of each center (in the $G_4$-orbit $I_4<0$) is given as the requirement that a given charge-dependent, Jordan-algebra valued matrix be positive definite.

\begin{figure}[h]
\begin{center}
\centerline{\includegraphics[width=0.9\textwidth]{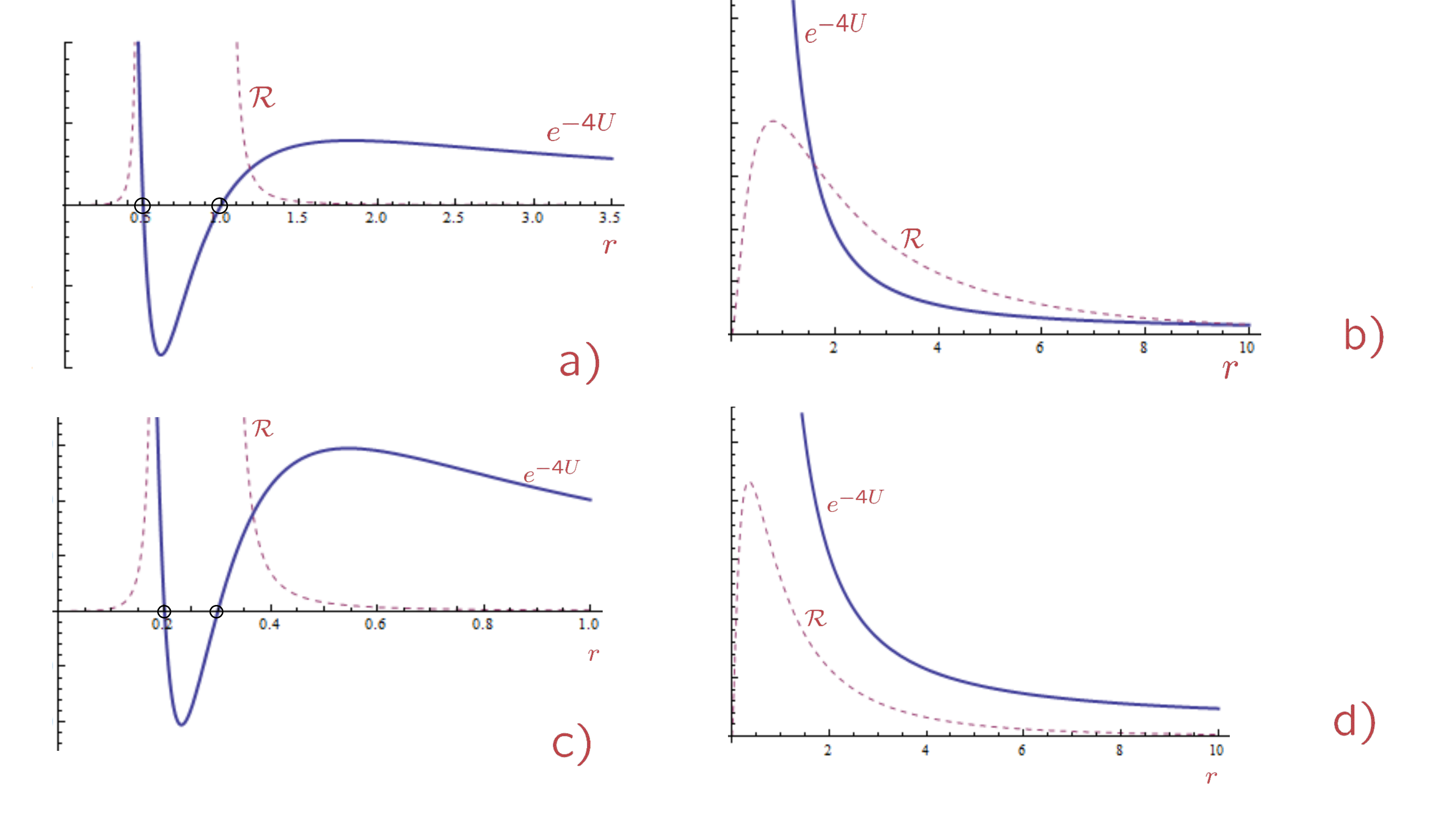}}
 \caption{\small Behaviour of the warp function $e^{-4U}$ and of the Ricci scalar $\mathcal{R}$ against $r$ for representatives of selected real orbits.  Fig.\ a): representative of the $(\gamma^{(6;1)},\,\beta^{(6;2)})$ orbit corresponding to a singular BPS solution (singularities being marked by a circle); Fig.\ b): representative of the $(\gamma^{(6;1)},\,\beta^{(6;1)})$ orbit corresponding to a regular BPS solution; Fig.\ c): representative of the $(\gamma^{(6;5)},\,\beta^{(6;5)},\,\delta^{(2)})$ orbit corresponding to a singular non-BPS solution (singularities being marked by a circle); Fig.\ d): representative of the $(\gamma^{(6;5)},\,\beta^{(6;5)},\,\delta^{(1)})$ orbit corresponding to a regular non-BPS solution of type iii). Solutions a) and b)  belong to the same complex orbit defined by $\gamma^{(6;1)}$, and similarly c) and d)  belong to the same complex orbit defined by $\gamma^{(6;5)}$. Clearly the values of $\mathcal{R}$ and of $e^{-4U}$ refer to different scales. They are plotted in the same graphs to illustrate the corresponding behaviours at the same values of $r$.}\label{fig1}
\end{center}
\end{figure}

Real orbits with $\alpha$-label ranging from $\alpha^{(1)}$ to $\alpha^{(5)}$ and $\gamma=\beta$ describe \emph{small} black holes, namely extremal single center solutions with vanishing horizon area. Such solutions are classically singular, though the singularity coincides with the center ($r=0$). Solutions in orbits with $\gamma\neq\beta$, on the other hand, just as for the  $\alpha^{(6)}$-orbits, feature singularities at finite nonzero $r$. Small black hole solutions were classified in \cite{Borsten:2011ai}: $\alpha^{(1)},\,\gamma^{(1;1)}=\beta^{(1;1)}$ corresponds to the \emph{doubly critical} solutions; $\alpha^{(2)},\,\gamma^{(2;\ell)}=\beta^{(2;\ell)}$, $\alpha^{(3)},\,\gamma^{(3;\ell)}=\beta^{(3;\ell)}$  and $\alpha^{(4)},\,\gamma^{(4;\ell)}=\beta^{(4;\ell)}$ to the \emph{critical} solutions;  $\alpha^{(5)},\,\gamma^{(5;\ell)}=\beta^{(5;\ell)}$ to the \emph{light-like} small solutions. The first representatives (corresponding to $\ell=1$) of each set of orbits describe solutions preserving an amount of supersymmetry (BPS solutions)\par
In \cite{mariodaniele} a composition rule of the 16 real orbits describing regular (small and large) single-center solutions into the higher order $\alpha^{(7)}-\alpha^{(11)}$-orbits is defined: generic representatives $e$ and $e'$ of two regular-single-center orbits $\mathcal{O}$ and $\mathcal{O}'$ are combined into a nilpotent representative of a higher order orbit $\mathcal{O}''$, under the general assumption that the corresponding neutral elements $h$ and $h'$ commute. A composition law is defined $\mathcal{O},\,\mathcal{O}'\rightarrow\mathcal{O}''$ and it is observed that some of the $\alpha^{(7)}-\alpha^{(11)}$-orbits are never obtained in this way. Such orbits are characterized as \emph{intrinsically singular} in that they contain no regular solution. They are represented by empty slots in Figure \ref{Figfigs}.

\section{The mathematical framework}\label{secsetup}

\noindent  In this section we define a symmetric pair $(\g,\g_0)$ and a class of Lie groups $G_0$ acting on $\g_1$; we also introduce the main example motivated in Section \ref{secourex}. The notation introduced in this section is retained throughout the paper.

\subsection{The symmetric pair $(\g,\g_0)$}
Let $\g$ be a semisimple Lie algebra over the real numbers. We assume throughout
that $\g$ is {\em split}, that is, it has a Cartan subalgebra $\h$ that is split over
the reals;  every complex semisimple Lie algebra contains such a split real form (see \cite[Corollary 6.10]{knapp}). Let $\varphi \colon \g \to \g$ be an automorphism of order 2 with eigenspace decomposition  $\g = \g_0 \oplus \g_1$, where $\varphi$ has eigenvalue $(-1)^i$ on $\g_i$. The pair $(\g,\g_0)$ is a {\em real symmetric pair}; note that the decomposition $\g = \g_0 \oplus \g_1$ is a $\Z/2\Z$-grading of $\g$, that is,
$[\g_i,\g_j] \subseteq \g_{i+j\bmod 2}$ for $i,j\in \{0,1\}$.

Let $\theta \colon \g\to \g$ be a Cartan involution of $\g$ commuting with
$\varphi$; such a Cartan involution exists and is unique up to conjugacy by an element in $\exp(\ad_\g( \g_0))$, see for example \cite[Theorem 1.1]{seki}.
Let $\g = \fk \oplus \fp$ be the Cartan decomposition associated with $\theta$; this decomposition is also a $\Z/2\Z$-grading of $\g$. Since any Cartan
subalgebra of $\g$ is conjugate to a $\theta$-stable Cartan subalgebra
(\!\cite[Proposition 6.59]{knapp}), we may assume that the split Cartan
subalgebra $\h$ of $\g$ is $\theta$-stable; this implies that $\h\subseteq
\fp$, because $\h$ is split.
Since $\theta$ and $\varphi$ commute, the two gradings are compatible, that is $\g_0 =(\g_0\cap \fk)\oplus (\g_0\cap \fp)$, and similarly for $\g_1$.

We exemplify the results of our paper by a detailed discussion of the following example, which is motivated by the discussion in Section \ref{secourex}.

\begin{example}\label{exa:1}
Let $\g$ be the real Lie algebra defined by
\[\g = \{ X\in \gl_8(\R) \mid X^\intercal M = -MX \}\quad\text{where}\quad  M = \left(\begin{smallmatrix} 0 & I_4 \\ I_4 & 0
\end{smallmatrix}\right),\]
with $I_4$ the $4\times 4$ identity matrix; this is the split Lie algebra of type $D_4$, see \cite[Theorem IV.9]{jac}. Here we consider the involution $\varphi\colon\g\to\g$, $X\mapsto DXD$, where $D=\diag(1,1,-1,-1,1,1,-1,-1)$. A Cartan involution
of $\g$  commuting with $\varphi$ is negative-transpose, that is, $\theta(X) = -X^\intercal$. These involutions can also conveniently be described by their action on a suitable generating set of $\g$. In the following let  $e_{ij}$ be the
$8\times 8$ matrix with a $1$ on position $(i,j)$ and zeros elsewhere; for  $i\in\{1,\ldots,4\}$ define $d_i = e_{ii}-e_{4+i,4+i}$, so that $\{d_1,\ldots,d_4\}$ spans a Cartan subalgebra of $\g$.  Now define  $h_1 = d_1-d_2$, $h_2=d_2-d_3$, $h_3=d_3+d_4$,
$h_4=d_3-d_4$, $e_1 = e_{12}-e_{65}$, $e_2=e_{23}-e_{76}$, $e_3=e_{38}-e_{47}$,
$e_4=e_{34}-e_{87}$, and, lastly, $f_i = e_i^\intercal$ for $1\leq i\leq 4$. A straightforward computation shows that these elements satisfy the following relations
\begin{equation*}
\begin{aligned}{}
[h_i,h_j]& = 0,\qquad & [h_i,e_j]& = C_{j,i} e_j, \\
[e_i,f_j]& = \delta_{ij} h_i,&[h_i,f_j]& = -C_{j,i}f_j,
\end{aligned}
\end{equation*}
where $\delta_{ij}$ is the Kronecker delta and $C_{i,j}$ is the entry $(i,j)$ of the Cartan matrix $C$ of the
root system of type $D_4$ with Dynkin diagram
\begin{center}
\scalebox{1.4}{\begin{picture}(70,17)
  \put(18,0){\circle{6}}
\put(15,4){{\tiny 1}}
  \put(33,0){\circle{6}}
  \put(48,0){\circle{6}}
  \put(33,15){\circle{6}}
\put(45,4){{\tiny 4}}
\put(29,4){{\tiny 2}}
\put(36,15){{\tiny 3}}
  \put(21,0){\line(1,0){9}}
  \put(36,0){\line(1,0){9}}
  \put(33,3){\line(0,1){9}}
\end{picture}}\label{d4}
\end{center}
By \cite[\S IV.3]{jac}, the set $\{h_i,e_i,f_i\mid i\in\{1,\ldots,4\}\}$ is a {\em canonical generating set} of
$\g$; in particular, an automorphism of $\g$ is uniquely
determined by its values on these elements. It follows readily from the definition that $\theta(e_i) =-f_i$,
$\theta(f_i) = -e_i$, and $\theta(h_i) = -h_i$ for all $i$; moreover,  $\varphi(e_i) =e_i$, $\varphi(f_i) = f_i$ for $i\neq 2$, $\varphi(e_2)=-e_2$, $\varphi(f_2) = -f_2$, and $\varphi(h_j) = h_j$ for all $j$. It follows that  $\dim \fk =\dim \g_0 = 12$ and $\dim \fp = \dim \g_1=16$. It is straightforward to work out bases for these subspaces; for example,
$\g_0$ is spanned by $\{h_i,e_i,f_i\mid i=1,3,4\}$ along with $e_0 = e_{16}-e_{25}$,
$f_0 = e_0^\intercal$, and $h_0 = d_1+d_2$. It follows that  $\g_0$ is isomorphic to the
direct sum of four copies of $\sl_2(\R)$.\exdone
\end{example}

Here and in the sequel we denote by  $\g^c=\C\otimes_\R \g$ the complexification of $\g$ and by $\ad \colon \g^c \to \End(\g^c)$ its adjoint map, that is, $\ad(x)\colon \g^c\to\g^c$, $y\mapsto [x,y]$. If we use the adjoint map of a different Lie
algebra, then we use a subscript, for example $\ad_\g\colon \g\to\End(\g)$. Note that $\varphi$ lifts to an involution of $\g^c$, with eigenspace
decomposition $\g^c = \g_0^c \oplus \g_1^c$.

\subsection{The Lie groups $G_0$}
We continue with the notation of the previous section, and  denote by $G^c$ the adjoint group of $\g^c$.
This group can be characterised in various ways: It is the connected
algebraic subgroup of $\GL(\g^c)$ with Lie algebra $\ad(\g^c)$ (see \cite[\S 1.2]{colmcgov}); it is also the connected component of the automorphism group of $\g^c$, and generated by inner automorphisms $\exp( \ad (x))$ with $x\in \g$ (see \cite[(I.7)]{onishchik}). In any way,  $G^c$ is a subgroup of the automorphism group of $\g^c$. We denote by $G_0^c$  the connected algebraic subgroup of $G^c$ with Lie algebra
$\ad (\g_0^c)$; alternatively, this is the subgroup of $G^c$ generated by all
$\exp( \ad (x))$ with $x\in \g_0^c$.

 Let $\mathcal{N}_1^c$ be the set of nilpotent elements in $\g_1^c$; the determination of $G_0^c$-orbit representatives in $\mathcal{N}_1^c$ has been discussed in the literature and we recall some of
the main results in Sections \ref{secsl2comp} and \ref{seccacomp}. Our main focus here is the determination of representatives in $\mathcal{N}_1$, the set of nilpotent elements in $\g_1$, under the action of a suitable group $G_0$. There are different interesting choices for $G_0$. For example, one could define $G_0$ as the adjoint group of $\g_0$, which is the analytic subgroup of $\GL(\g_0)$ with Lie algebra $\ad_\g(\g_0)$ (see \cite[\S II.5]{helgason}). Alternatively, one could define $G_0$ as the set of real points $G_0^c(\R)$, that is, the subgroup of  elements of $G_0^c\subseteq \GL(\g^c)$ whose matrix (with respect to some fixed basis of $\g^c$ consisting of elements of $\g$) has real entries only.
We aim to provide a framework which allows us to deal with several different choices of $G_0$.

More precisely, here we define a group $G_0$ acting on $\g_1$ as follows; we exemplify our construction in  Example \ref{exa:2} below. We start with an isomorphism of algebraic groups $R^c\colon \widetilde{G}_0^c\to G_0^c$, where $\widetilde{G}_0^c$ is a connected algebraic subgroup of $\GL_k(\C)$ for some $k$;
we also assume that both $\widetilde{G}_0^c$ and $R^c$ are defined over $\R$.
We define $\widetilde{G}_0 =\widetilde{G}^c_0(\R)$ as the group consisting of all $g\in  \widetilde{G}^c_0$ with coefficients in $\R$, and then define  \[G_0 = R^c(\widetilde{G}^c_0(\R)).\]  It follows from \cite[\S 7.3]{borel} that the isomorphism $R^c$ induces a (not necessarily surjective) embedding of  $\widetilde G_0^c(\R)$ into $G_0^c(\R)$ which maps the identity component of  $\widetilde G_0^c(\R)$ onto that of $G_0^c(\R)$; moreover  $G_0$ is closed and has
 finite index in $G_0^c(\R)$. We note that the group $G_0^c(\R)$ has finitely many connected components, cf.\ \cite[p.\ 276 (c)(i)]{borel2}.  In conclusion, we have \[(G_0^c(\R))^\circ \leq G_0\leq G_0^c(\R)\leq \Aut(\g)\leq\Aut(\g^c).\]
Here $(G_0^c(\R))^\circ$ denotes the identity component  of $G_0^c(\R)$ (in the real Euclidean
topology), which is the same as the identity component of
$G_0$. The groups  $(G_0^c(\R))^\circ$, $G_0$, and $G_0^c(\R)$ all have the same
Lie algebra, which is $\ad_\g (\g_0)\cong \g_0$.

Recall that a Lie algebra is reductive if it is the direct sum of its semisimple derived subalgebra and its center. It is well-known that $\g_0$ is a reductive Lie algebra; more precisely, it is ``reductive in $\g$'' (see for example \cite{dfg2}). There exist different definitions for a Lie group to be reductive; here we use  the quite technical definition given in  \cite[Section VII.2]{knapp}), mainly because this allow us to use the results in  \cite[Chapter VII]{knapp} for reductive groups.

\begin{definition}\label{def:red}
A real Lie group $\G$ is reductive if there is a quadruple $(\G,\K,\eta,\B)$, where
$\K\leq \G$ is a compact subgroup, $\eta$ is an involution
of the Lie algebra $\mc{g}$ of $\G$, and $\B$ is a nondegenerate $\Ad(\G)$-
and $\eta$-invariant bilinear form on $\mc{g}$, such that the following hold:
\begin{items}
\item[(i)] $\mc{g}$ is reductive,
\item[(ii)] the $\pm1$-eigenspace decomposition of  $\eta$ is $\mc{g}=\mc{k}\oplus \mc{p}$, where $\mc{k}$ is the
Lie algebra of $\K$,
\item[(iii)] $\mc{k}$ and $\mc{p}$ are orthogonal under $\B$, and $\B$ is negative
definite on $\mc{k}$ and positive definite on $\mc{p}$,
\item[(iv)] the multiplication map $\K\times \exp(\mc{p}) \to \G$ is a
surjective diffeomorphism,
\item[(v)] for each $g\in \G$, the automorphism $\Ad(g)$ of $\mc{g}^c$ is inner, that is, it lies in
$\Int(\mc{g}^c)$.
\end{items}
\end{definition}
If $\G$ is a closed linear Lie group, then $\Ad(g)(x) =
gxg^{-1}$ for $x\in \mc{g}$ and $g\in\G$, see \cite[p.\ 79]{knapp}. Recall that $\Int(\mc{g}^c)$ is the analytic
subgroup of $\Aut(\mc{g}^c)$ with Lie algebra $\ad(\mc{g}^c)$, generated by all $\exp( \ad(x))$ with $x\in \mc{g}^c$, see \cite[p.\ 1]{onishchik}. It contains the connected algebraic subgroup of
$\Aut(\mc{g}^c)$ with Lie algebra $\ad(\mc{g}^c)$.

\begin{proposition}\label{prop:red}
The group $G_0$ is a reductive Lie group.
\end{proposition}

\begin{proof}
We first show that $G=G^c(\R)$ is reductive. Firstly, $G$ has Lie algebra $\ad_\g(\g)$, which is semisimple, hence reductive. It follows from \cite[\S5.(5)]{onishchik} that  $\ad_\g(\g)$ has the inner product $(\ad_\g(x),\ad_\g(y))=-\kappa(x,\theta(y))$, where $\kappa$ is the Killing form of $\g^c$.
Let $O(\ad_\g(\g))$ be the group of all bijective endomorphisms of $\ad_\g(\g)$ that leave this
inner product invariant, and  define  $K=\{g\in G\mid \Ad(g)\in O(\ad_\g(\g))\}$. Define a bilinear form $B$ on $\ad_\g(\g)$  by $B(\ad\nolimits_\g(x),\ad\nolimits_\g(y))=\kappa(x,y)$. We extend the Cartan involution $\theta$ of $\g$ to an automorphism of
$\ad_\g (\g)$ by setting  $\theta(\ad_\g (x)) = \ad_\g (\theta(x))$.  We claim
that $(G,K,\theta,B)$ satisfies Definition \ref{def:red}. Similarly to
$\ad_\g(\g)$, the Lie algebra $\g$ has an inner product defined by $(x,y) =
-\kappa(x,\theta(y))$. Let $O(\g)$ be the group of bijective endomorphisms
of $\g$ leaving this inner product invariant. Using \cite[Lemma 1.118]{knapp},
one sees that $K= G\cap O(\g)$,
and as $K$ is closed, it follows that it is compact.
By \cite[Proposition 1.119]{knapp},  the Killing form is invariant under
$\Aut(\g^c)$, which implies that $B$ is $\Ad(G)$- and $\theta$-invariant; since $\kappa$ is nondegenerate on $\g$, so is $B$. Clearly,  $\theta$ is a Cartan involution with Cartan decomposition $\ad_\g(\g)=\ad_\g(\fk)\oplus\ad_\g(\fp)$.  Since the latter is  a $\Z/2\Z$-grading, $\ad_\g(\fk)$ and $\ad_\g(\fp)$ are orthogonal under
$B$. Moreover, $B$ is positive definite on $\ad_\g(\fp)$ and negative
definite on $\ad_\g(\fk)$ by \cite[\S5 (5) \& (6)]{onishchik}. Note that the Lie algebra of $G$ is semisimple, so that $G$ is semisimple (see \cite[p.\ 56]{onvi}). This allows us to apply \cite[\S 5.3, Theorem 2 \& Corollary 2]{onvi}, which proves that (iv) of Definition \ref{def:red} holds, and that the Lie algebra of $K$ is $\ad_\g(\fk)$.  To establish (v) of Definition \ref{def:red}, consider $g\in G$. Since $G$ is a closed linear group,  $\Ad(g)(\ad_\g(x)) = g(\ad_\g(x))g^{-1}$ for all $x\in\g$ by \cite[p.\ 79]{knapp}; now \cite[Lemma 1.118]{knapp} shows that $g(\ad_\g(x))g^{-1}=\ad_\g g(x)$, which proves that $\Ad(g)$ is induced by the automorphism $g$; as $G\leq G^c\leq \Int(\g^c)$, we obtain (v).

Now we consider $G_0$. By abuse of notation we also use the symbols
$\theta$ and $B$ to denote
their restrictions to $\ad_\g(\g_0)$. Define $K_0=K\cap G_0$ where $K$ is as defined in the discussion of $G$ above; we claim
that $(G_0,K_0,\theta,B)$ satisfies Definition \ref{def:red}. Clearly, $G_0$ is a real Lie group with reductive Lie algebra $\ad_\g(\g_0)$.
Write $\fk_0=\g_0\cap\fk$ and $\fp_0=\g_0\cap\fp$, and note that
$\ad_\g(\g_0)=\ad_\g(\fk_0)\oplus\ad_\g(\fp_0)$ is the $\pm1$-eigenspace
decomposition of $\theta$; as before, this decomposition is orthogonal with
respect to $B$, and $B$ is positive definite on $\ad_\g(\fp_0)$  and negative
definite on $\ad_\g(\fk_0)$. Since the Killing form is nondegenerate on $\g$, it follows that $B$ is nondegenerate on $\ad_\g(\g_0)$. Since $B$ is $\text{Ad}(G)$-invariant, it is also $\text{Ad}(G_0)$-invariant; clearly, $B$ is $\theta$-invariant. Let $x\in \g_0^c$, $g=\exp(\ad(x))\in G_0^c$, and $\hat g =
\exp( \ad_{\g_0^c}(x))$. Then the restriction of $\Ad(g)$ to $\ad(\g_0^c)$ is
equal to $\Ad(\hat g)$. It follows that $g\in \Int(\g_0^c)$.
The group $G_0^c$ is generated by all $\exp( \ad(x))$ with $x\in \g_0^c$, thus $G_0^c \leq \Int(\g^c)$ and each  $\Ad(g)$ with $g\in G_0$ lies in $\Int(\g^c)$; this establishes (v) of Definition \ref{def:red}. Now we consider (iv) of Definition \ref{def:red}. Conjugation by $\varphi$
is an automorphism of $\Aut(\g^c)$, so it stabilises the identity component $G^c=\Aut(\g^c)^\circ$. Since $\varphi$ is defined over $\R$, conjugation by it is an automorphism of $G=G^c(\R)$. Furthermore, for $x\in \g_0$ we have
$\varphi \exp( \ad(x))\varphi^{-1} = \exp( \ad(\varphi(x)) ) = \exp( \ad (x))$,
so that conjugation by $\varphi$ is the identity on $G_0$.
We consider $\varphi$ as an automorphism of $\ad_\g(\g)$ via $\varphi(\ad_\g(x))=\ad_\g(\varphi(x))$. Since $\varphi$ and $\theta$ commute, the inner product of $\ad_\g(\g)$ defined above satisfies \[(\varphi(\ad\nolimits_\g(x)),\varphi(\ad\nolimits_\g(y)))=-\kappa(\varphi(x),\theta(\varphi(y)))=-\kappa(\varphi(x),\varphi(\theta(y)))= (\ad\nolimits_\g(x),\ad\nolimits_\g(y)),\]
where the last equation follows from \cite[Lemma 1.119]{knapp}. This implies that if $k\in K=G\cap O(\g)$, then also $\varphi k \varphi^{-1}\in K$. Now let $g\in G_0$; since $g\in G$, we can write $g=k \exp(\ad_\g(x))$ for uniquely determined $k\in K$ and $x\in \fp$, cf.\ part (iv) for the reductive tuple $(G,K,\theta,B)$.  Because $g= \varphi g\varphi^{-1}$ we have
\[k\exp(\ad\nolimits_\g(x))=\varphi k \varphi^{-1} \varphi \exp (\ad\nolimits_\g(x)) \varphi^{-1} = \varphi k \varphi^{-1} \exp(\ad\nolimits_\g(\varphi(x))).\]

As $\varphi k \varphi^{-1}\in K$ and $\varphi(x)\in \fp$, we conclude that
$\varphi k \varphi^{-1}=k$ and $\varphi(x)=x$ by uniqueness. In particular, $x\in \fp_0$, so that $\exp( \ad_\g(x))\in G_0$ and therefore also $k\in G_0$.
As $K_0\subseteq K$ is closed, the group $K_0$ it is compact, and the Lie algebra of $K_0$ is the intersection of the Lie algebras of $K$ and $G_0$
(this follows immediately from the standard definition of the Lie algebra of
a linear Lie group, see \cite{hine}, Definition 4.1.3),
therefore, it is $\ad_\g(\fk_0)$.
\end{proof}

\begin{example}\label{exa:2}
We continue with the notation of Example \ref{exa:1}, and we denote the basis elements of $\g^c$
by the same symbols as for $\g$. Let $\tilde\g_0^c$ be the direct sum
of four copies of $\sl_2(\C)$, seen as a subalgebra of $\gl_{8}(\C)$ in the
natural way, that is, the elements of $\tilde\g_0^c$ are block-diagonal matrices,
where each block is of size $2\times 2$ and corresponds to a copy of
$\sl_2(\C)$. Let $\{\tilde h_i, \tilde e_i, \tilde f_i\mid i=0,1,3,4\}$ be
a basis of $\tilde \g_0^c$ such that  $r^c \colon \tilde\g_0^c\to
\g_0^c$ with $r^c(\tilde e_i) =e_i$ and $r^c(\tilde f_i) = f_i$ defines an isomorphism. Let $\widetilde{G}^c_0$ be the direct product of four copies of $\SL_2(\C)$, embedded in $\GL_{8}(\C)$ in the same fashion as $\tilde\g_0^c$ is embedded in
$\gl_{8}(\C)$, so that the Lie algebra of $\widetilde{G}^c_0$ is $\tilde\g_0^c$. By construction, there is a surjective homomorphism of algebraic groups
$R^c \colon \widetilde{G}^c_0 \to G_0^c$ satisfying $R^c(\exp( u )) =
\exp( \ad (r^c(u)))$ for all $u\in \tilde \g_0^c$, cf.\ \cite[\S 5, Corollary 1]{steinb}. Note that $\widetilde{G}_0=\widetilde{G}_0^c(\R)$, which is the  direct product of four copies of $\SL_2(\R)$,
and we set $G_0 = R^c(\widetilde{G}_0)$. Let $V_2^c$ be the natural $\SL_2(\C)$-module. We write an element of $\widetilde{G}_0^c$ as a tuple $(g_1,g_2,g_3,g_4)$, where $g_i$ lies in the $i$-th copy of $\SL_2(\C)$. Define $\rho^c \colon \widetilde{G}_0^c \to \GL(V_2^c\otimes V_2^c\otimes V_2^c\otimes V_2^c)$
by $\rho^c(g_1,\ldots,g_4)(v_1\otimes \cdots \otimes v_4) = (g_1v_1)\otimes
\cdots \otimes (g_4v_4)$. Similarly, let $V_2$ be the natural $\SL_2(\R)$-module
and define $\rho \colon\widetilde{G}_0 \to \GL(V_2\otimes V_2\otimes V_2\otimes
V_2)$ in the analogous way. By comparing weights
one sees that $\rho^c$ is
equivalent to the representation $\widetilde{G}_0^c \to G_0^c \to \GL(\g_1^c)$. Note that there are bases of $V_2^c\otimes \ldots\otimes V_2^c$ and $\g_1^c$ such that the matrix of a given element of $\widetilde G_0^c$ is the same with respect to both bases; these bases have coefficients in $\R$. Thus, everything is defined over $\R$, and it follows that $\rho$ is equivalent to the representation $\widetilde{G}_0 \to G_0 \to \GL(\g_1)$. In conclusion, determining the
orbits in $V_2\otimes\ldots\otimes V_2$ under the action of $\widetilde{G}_0=\SL_2(\R)\times\ldots\times \SL_2(\R)$ is equivalent to determining the
$G_0$-orbits in $\g_1$. Lastly, we note that $G_0\ne G_0^c(\R)$:  For $t\in\C$ let
$w_1(t) = \exp( t\tilde e_1) \exp(-t^{-1} \tilde f_1 )\exp(t\tilde e_1)$, and
$a= w_1(\imath) w_1(1)^{-1}$, where $\imath$ denotes the imaginary unit.
Similarly we define $w_4(t)$ and $b=w_4(\imath)w_4(1)^{-1}$. Then
$a,b\in \widetilde{G}_0^c$, and $R^c(ab)$ is an automorphism of $\g^c$ defined
over $\R$. So $R^c(ab)\in G_0^c(\R)$. But $R^c(ab)$ does not lie in
$G_0$.

As basis of $V_2\otimes\ldots\otimes V_2$ we take the  eight vectors of the form $v_1^{(\pm)}\otimes v_2^{(\pm)}\otimes v_3^{(\pm)}\otimes v_4^{(\pm)}$, where each $v_k^{(+)}=(1,0)$ and $v_k^{(-)}=(0,1)$ are eigenvectors of the Cartan subalgebra generator ${\rm diag}(1,-1)$ of the $k$-th copy of $\sl_2(\R)$. Moreover, we  adopt the following short-hand notation and write
\[
(\pm,\pm,\pm,\pm)\equiv v_1^{(\pm)}\otimes v_2^{(\pm)}\otimes v_3^{(\pm)}\otimes v_4^{(\pm)}.
\]

\vspace*{-0.8cm}
~ \exdone

\end{example}

\section{Nilpotent orbits from $\sl_2$-triples}\label{triples}
\noindent We continue with  the notation introduced in Section \ref{secsetup}; recall that an $\sl_2$-triple in $\g^c$ is a tuple $(h,e,f)$ of elements in $\g^c$ satisfying $[h,e]=2e$, $[h,f]=-2f$, and $[e,f]=h$.  In this section we describe how $\sl_2$-triples can be used to investigate the nilpotent orbits of a symmetric pair. Although our main concern lies with
real symmetric pairs, we start with some remarks regarding the complex case.
We do this for two reasons: firstly because some of those remarks are
also needed when we deal with real symmetric pairs, secondly because the
starting point of the method that we propose for the real case is the list of
complex nilpotent orbits.

\subsection{The complex case}\label{secsl2comp}
We start with a lemma, proved in \cite[Proposition 4 and Lemma 4]{kora}.

\begin{lemma}\label{lem:sl2C}
 Every nonzero nilpotent $e\in \g_1^c$ can be embedded in a homogeneous $\sl_2$-triple $(h,e,f)$, that is, $h\in \g_0^c$ and $e,f\in \g_1^c$. If  $(h,e,f)$ and $(h',e',f')$ are homogeneous $\sl_2$-triples, then $e$ and $e'$
are $G_0^c$-conjugate if and only if the two triples are $G_0^c$-conjugate,
if and only if $h$ and $h'$ are $G_0^c$-conjugate.
\end{lemma}
If $h\in \g_0^c$ lies in an $\sl_2$-triple $(h,e,f)$, then $h$ is called
the {\em characteristic} of triple; this terminology goes back to Dynkin \cite{dyn}. We say that
$h\in \g_0^c$ is a {\em (homogeneous) characteristic} if it is the characteristic of some (homogeneous)
$\sl_2$-triple $(h,e,f)$. Lemma \ref{lem:sl2C} shows that the classification of
the nilpotent $G_0^c$-orbits in $\g_1^c$ is equivalent to the classification
of the $G_0^c$-orbits of homogeneous characteristics in $\g_0^c$.

Let $\h_0^c$ be a fixed Cartan subalgebra of $\g_0^c$; recall that all Cartan subalgebras of $\g_0^c$ are $G_0^c$-conjugate. Since every characteristic in $\g_0^c$ lies in some Cartan subalgebra of $\g_0^c$, this implies that  every $G_0^c$-orbit of homogeneous characteristics has at least one element in $\h_0^c$.
Furthermore,  two elements of $\h_0^c$ are $G_0^c$-conjugate if and only if they
are conjugate under the Weyl group $W_0$ of the root system of $\g_0^c$ with respect to $\h_0^c$, see \cite[Theorem 2.2.4]{colmcgov} and the remark below that theorem.

The approach now is to start with the classification of the nilpotent $G^c$-orbits in $\g^c$; this classification is well-known, see for example \cite{colmcgov}. Note that $\h_0^c$ is contained in a Cartan subalgebra of $\g^c$ (see for example \cite[Lemmas 8 \& 16)]{dfg2}). In the following, for simplicity, assume that $\h_0^c=\h^c$ is a Cartan subalgebra of $\g^c$; note that this holds in our main example (Example \ref{exa:1}). The classification of the
nilpotent $G^c$-orbits in $\g^c$ yields a finite list of $\sl_2$-triples
$(\hat h_i,\hat e_i,\hat f_i)$, $i=1,\ldots,m$, such that each $\hat h_i\in \h_0^c$
and $\{\hat e_i\mid i=1,\ldots,m\}$ is a list of representatives of the nilpotent orbits.
Let $W$ be the Weyl group of
the root system of $\g^c$ with respect to $\h_0^c$. The union
$\mathcal{H}$ of the orbits $W\cdot \hat h_i$ for $1\leq i\leq m$ is exactly the set of elements of $\h_0^c$
lying in an $\sl_2$-triple. Let $W_0\leq W$ denote the Weyl group of the root system
of $\g_0^c$ with respect to $\h_0^c$. Acting with $W_0$, we now reduce the set $\mathcal{H}$ to a subset $\mathcal{H}'$ of $W_0$-orbit representatives.  For each
$h\in \mathcal{H}'$ we then decide whether it is a homogeneous characteristic; we omit the details here and refer to \cite{gra15} instead. For each such $h$ we obtain a homogeneous $\sl_2$-triple $(h,e,f)$; the set of the elements
$e$ obtained in this way is a complete and irredundant list of representatives of the
nilpotent $G_0^c$-orbits in $\g_1^c$. Again, we refer to \cite{gra15} for further
details, and for an account of what happens when $\h_0^c$ is not a Cartan
subalgebra of $\g^c$.

\subsection{The real case} \label{sec:realsl2}
Now we consider the real case. Also here, for a nonzero nilpotent $e\in \g_1$, there exist $h\in \g_0$ and $f\in \g_1$ such that $(h,e,f)$ forms an $\sl_2$-triple; this follows from \cite[Proposition 4]{kora}, see also \cite[Theorem 2.1]{vanle}. In analogy to the previous
subsection, we say that $h$ is a homogeneous characteristic.

Recall that every Cartan subalgebra of $\g_0$ lies in some Cartan subalgebra of $\g$, and that $\g$ is assumed to have a Cartan subalgebra which is split over the reals. By \cite[Proposition 6.59]{knapp}, every Cartan subalgebra of $\g$ is $G$-conjugate to a $\theta$-stable Cartan subalgebra. This shows that there exists a Cartan subalgebra $\h_0$ of $\g_0$ which is split and $\theta$-stable. It follows from \cite[Lemma 11]{dfg2} that
\begin{eqnarray}\label{eqh0}\h_0 \subseteq \g_0\cap \fp;
\end{eqnarray}
see also  \cite[Chapter 4, Proposition 4.1]{vinbergiv}.
 Let $\h\leq \g$ be a split and $\theta$-stable Cartan subalgebra of $\g$ containing $\h_0$; as usual, $\h^c$ is the complexification of $\h$.

Using the methods indicated in Section \ref{secsl2comp}, we can determine representatives $e_1,\ldots,e_t$ of the nonzero $G_0^c$-orbits of nilpotent elements in $\g_1^c$, with corresponding homogeneous triples $(h_i,e_i,f_i)$ in $\g^c$.
We may assume that each $h_i\in \h_0$ (this follows from the fact that $\ad(h_i)$ has rational eigenvalues, and so $h_i$ is a rational linear combination of
the semisimple elements of a Chevalley basis of $\g$ with respect to $\h$).
If $e\in\g_1$ is non-zero and nilpotent, then, as mentioned above, $e$ lies in some homogeneous $\sl_2$-triple $(h,e,f)$ of $\g$ with $h\in\h_0$; since this triple can also be considered as a complex $\sl_2$ triple in $\g^c$, the discussion in Section \ref{secsl2comp} shows that $h$ is $W$-conjugate to some $h_j\in \h_0$.

A homogeneous $\sl_2$-triple $(h,e,f)$ in $\g$ is a  {\em real Cayley triple}
if  $\theta(e)=-f$, hence  $\theta(f)=-e$ and $\theta(h)=-h$.

\begin{proposition}\label{prop:sl21}
Let $h\in \g_0$ be a homogeneous characteristic, lying in a homogeneous
$\sl_2$-triple $(h,e,f)$. Then there is $g\in G_0$ and $i\in\{1,\ldots,t\}$ such that
$g(h) = h_i$ and $(g(h),g(e),g(f))$ is a real Cayley triple.
\end{proposition}

\begin{proof}
From Proposition \ref{prop:red} we have that $G_0$ is reductive, and we let $K_0$ be
the compact subgroup of $G_0$, as in the proof of that proposition. Recall that $G_0$ contains the analytic subgroup of $G$ with Lie algebra $\ad_\g (\g_0)$. Now \cite[Lemma 1.4]{seki} shows that there is a $g'\in G_0$ such that $(h',e',f')=(g'(h),g'(e),g'(f))$ is a real Cayley triple.
As $\ad (h')$ has integral eigenvalues, it follows from \cite[Lemma 11]{dfg2} that $h'\in \g_0\cap \fp$; in particular,  $h'$ lies in a  maximal $\R$-diagonalisable subspace ({\em Cartan subspace}) $\fa$ of
$\g_0\cap \fp$. We have seen in \eqref{eqh0} that $\h_0$ is also a Cartan
subspace of $\g_0\cap \fp$, and now  \cite[Proposition 7.29]{knapp} asserts that there is a $k'\in K_0$ with $k'(\fa) =\h_0$. We define $h'' = k'(h')$, so
that $h''\in \h_0$. By \cite[Proposition 7.32]{knapp}, the group
$N_{K_0}(\h_0)/Z_{K_0}(\h_0)$ coincides with the Weyl group of the root system
of $\g_0$ with respect to $\h_0$ (note that we assume that $\h_0$ is split,
therefore, in this case, ``restricted roots'' are the same as ``roots''). Together, in view of what is said in the
first subsection, there is a $k''\in K_0$ such that
$k''(h'')=h_i$ is one of the fixed characteristics $h_1,\ldots,h_t\in\h_0$
from the complex case. Now define $k=k''k'$; then $k(h')=h_i$, and, as
$k\in K_0$ commutes with $\theta$ (see \cite[Proposition 7.19(c)]{knapp}),
it follows that  $(k(h'),k(e'),k(f'))$ is a real Cayley triple.
\end{proof}

Based on this proposition we outline a procedure for obtaining a list of
nilpotent elements of $\g_1$, lying in real Cayley triples, such that
each nilpotent $G_0$-orbit in $\g_1$ has a representative in this list. At this stage, there may be $G_0$-conjugate elements in that list; we explain in a later section how to deal with that. First we give a rough outline of the main steps of our procedure; we comment on these steps subsequently.

\begin{enumerate}
\item Using the algorithms sketched in Section \ref{secsl2comp}, compute a list of
representatives of the nilpotent $G_0^c$-orbits in $\g_1^c$, by obtaining
$\sl_2$-triples $(h_i,e_i,f_i)$, $i=1,\ldots,s$, with each $h_i\in \h_0$.
\item For each $h\in\{h_1,\ldots,h_s\}$ do the following:
\begin{enumerate}
\item Compute a basis $\{u_1,\ldots, u_d\}$ of $\g_1(2)=
\{ x\in \g_1 \mid [h,x] = 2x \}$.
\item Compute polynomials $p_1,\ldots,p_m\in \R[T_1,\ldots,T_d]$ such
that $p_i(c_1,\ldots,c_d)=0$ for all $i$ is equivalent to
$(h,e,-\theta(e))$ with $e=\sum_i c_i u_i$ being a real Cayley triple.
\item Describe the variety $C\subseteq \g_1(2)$ defined by the  polynomial equations $p_1=\ldots=p_m=0$.
\item Compute  $\fz = \{ x\in
\fk\cap\g_0\mid [h,x]=0\}$ and act with $\exp(\ad_\g(k))\in G_0$, where $k\in \fz$, to get rid of  as many $G_0$-conjugate copies in $C$ as possible. The result is the list that remains.
\end{enumerate}
\end{enumerate}
The only steps that are not straightforward are the last two; we use the realisation of $\g$ as a Lie algebra of $n\times n$ matrices to find a description of the variety $C$ (see Example \ref{exa:2b} for an illustration).
Now in $C$ we need to find  non-conjugate elements. A first step towards
that is performed using explicit elements of $G_0$, namely by using elements from the group corresponding to the centraliser $\fz$. Note that
elements from this group leave $h$ invariant and map a real Cayley triple
to a real Cayley triple; in particular, they stabilise $C$. Finally, after having fixed the action of the centralizer of $h$ on $C$, we define a set of $G_0$-invariant quantities, to be discussed in detail in Section \ref{SNO}, such that if nilpotent generators in $C$ are not distinguished by values of these invariants, they are shown to be connected by $G_0$. The resulting classification is then complete.

\begin{example}\label{exa:2b}
We continue with the notation of Examples \ref{exa:1} and \ref{exa:2}, and realise  $\g$ as a Lie algebra of $8\times 8$ matrices. Since $\theta(X)=-X^\intercal$, the polynomial equations are equivalent to $[e,e^\intercal]=h$, where $e=\sum_i T_i u_i$; these equations are easily
written down. We then use the function {\tt Solve} of {\sc Mathematica} \cite{mathematica}; in the context of this example, this function always returned a
list of matrices, some of which depend on parameters; this serves as description of the variety $C$.  A basis of the centraliser $\fz$ is readily found by solving a set of linear equations. If $x\in \fz$, then $\exp(x)$ is an element of $K_0\leq G_0$. Since $x$ again is a matrix, the function {\tt MatrixExp} of {\sc Mathematica} can be used to compute $g=\exp(x)$. This element
acts on $C$ by $e \mapsto geg^{-1}$. By varying $x$ we managed to reduce the final list to a finite number of elements.

As an example, we consider the representative $h\in \mathfrak{h}_0$ of the $G^c_0$-orbit
defined by the $\gamma$-label $\gamma^{(3;1)}=(2,0,2,0)$, which belongs to the $G^c_0$ orbit with $\alpha$-label $\alpha^{(3)}=(0,0,2,0)$.
A generic element of the centraliser $Z(h)=\{\exp(x)\vert\,x\in \mathfrak{z}\}$ of $h$ has the form
\begin{equation}g=\left(\begin{smallmatrix}1 & 0\cr 0 & 1\end{smallmatrix}\right)\otimes \left(\begin{smallmatrix}\cos(\alpha_2) & \sin(\alpha_2)\cr -\sin(\alpha_2) & \cos(\alpha_2)\end{smallmatrix}\right)\otimes \left(\begin{smallmatrix}1 & 0\cr 0 & 1\end{smallmatrix}\right)\otimes \left(\begin{smallmatrix}\cos(\alpha_4) & \sin(\alpha_4)\cr -\sin(\alpha_4) & \cos(\alpha_4)\end{smallmatrix}\right)\,.
\end{equation}
Using the {\tt Solve} function we find, modulo an overall sign, the following parameter-dependent solutions to the equation $[e,e^T]=h$:
\begin{eqnarray}
e_1&=&\frac{1}{2}\sqrt{1-a^2}\,[(+,-,-,-)+(+,+,-,+)] +\frac{a}{2}[(+,+,-,-)-(+,-,-,+) ]\,,\nonumber\\[1ex]
e_2&=&\frac{1}{2}\sqrt{1-b^2}\,[(+,-,-,-)-(+,+,-,+)] +\frac{b}{2}[(+,+,-,-)+(+,-,-,+) ]\,,\nonumber
\end{eqnarray}
every concrete solution is defined by special choices of the above parameters. By acting with $g\in Z(h)$ we can transform $e_1$ and $e_2$ into the following elements:
 \begin{eqnarray}
e_1'&=&ge_1g^{-1}=\frac{1}{2}[(+,+,-,-)-(+,-,-,+) ]\,,\nonumber\\
e_2'&=&ge_1g^{-1}=\frac{1}{2}[(+,+,-,-)+(+,-,-,+) ]\,.\nonumber
\end{eqnarray}
The above representatives correspond to two distinct $\beta$-labels, and therefore  belong to distinct $G_0$-orbits.
\exdone
\end{example}

\section{Nilpotent orbits from carrier algebras}\label{carrier1}

\subsection{The complex case}\label{seccacomp}

Here we comment briefly on the problem of finding the $G_0^c$-orbits in
$\mathcal{N}_1^c$ using Vinberg's carrier algebra method, introduced in \cite{vinberg2}. For more details and proofs, we refer to
that paper.

This method roughly consists of two parts. In the first part, to each nonzero nilpotent $e\in \g_1^c$ a class of
subalgebras of $\g^c$ is associated; these subalgebras are called
{\rm carrier algebras} of $e$. Two different carrier algebras of $e$ are conjugate
under $G_0^c$, so up to conjugacy by that group, $e$ has a unique carrier algebra; so we speak of {\em the} carrier algebra of $e$. Moreover, it is shown
that two nonzero nilpotent elements of $\g_1^c$ are $G_0^c$-conjugate if and only if
their carrier algebras are $G_0^c$-conjugate. The second part is to show that
classifying carrier algebras boils down to classifying certain subsystems of the
root system of $\g^c$, up to the action of the Weyl group of $\g_0^c$. We give more details:

Let $e\in \g_1^c$ be nilpotent and nonzero. Each homogeneous $\sl_2$-triple  $(h,e,f)$ in $\g^c$ spans a subalgebra $\fa^c\leq \g^c$ isomorphic to $\sl_2(\C)$, and we denote by $\fz_0(\fa^c)  =\{x \in \g_0^c \mid [x,\fa]=0\}$ its centraliser in $\g^c$. We choose a Cartan subalgebra $\h_z$ of $\fz_0(\fa^c)$, and let \[\ft^c=\text{Span}_\C(h)\oplus \h_z\] be the
subalgebra spanned by $\h_z$ and $h$. Define the linear map $\lambda \colon \ft^c\to \C$ by $[t,e] = \lambda(t)e$ for $t\in \ft$, and for $k\in \Z$ set
$$\g_k^c(\ft^c,e) = \{ x\in \g^c_{k\bmod 2} \mid [t,x] = k\lambda(t) x
\text{ for all } t\in \ft^c\}.$$
Next define
$$\g^c(\ft^c,e) = \bigoplus\nolimits_{k\in \Z} \g_k^c(\ft^c,e).$$
The carrier algebra of $e$ is $\fs^c(e) = [\g^c(\ft^c,e),\g^c(\ft^c,e)]$ with the induced
$\Z$-grading $\fs^c(e)=\bigoplus_{k\in\Z} \fs^c(e)_k$, where $\fs^c(e)_k =
\fs^c(e)  \cap \g_k^c(\ft^c,e)$. Note that  $\fs^c(e)$ depends on $\h_z$; however, any two such Cartan subalgebras are conjugate under
$G_0^c$ (in fact, under the smaller group $Z_{G_0^c}(\fa^c)$), so
carrier algebras resulting from different choices of $\h_z$ are
$G_0^c$-conjugate. Carrier algebras posses some nice properties and $\fs^c(e)$ \ldots
\begin{items}
\item[(1)] is a semisimple and $\Z$-graded subalgebra, that is, each $\fs^c(e)_k\leq \g^c_{k\bmod 2}$,
\item[(2)]  is regular, that is, normalised by a Cartan subalgebra of $\g_0^c$,
\item[(3)]  is complete, that is, not a proper subalgebra
of a reductive $\Z$-graded subalgebra of the same rank,
\item[(4)]  is locally flat, that is, $\dim \fs^c(e)_0 = \dim \fs^c(e)_1$,
\item[(5)] has $e\in \fs^c(e)_1$ in general position, that is, $[\fs^c(e)_0,e]=
\fs^c(e)_1$.
\end{items}
In general, any subalgebra of $\g^c$ satisfying (1)--(4) is  called a carrier algebra of $\g^c$. The carrier algebra of a nilpotent element in $\g_1^c$ is a carrier algebra and, conversely, a carrier algebra $\fs^c\leq \g^c$ is a carrier algebra of $e$ where $e\in\fs^c_1$ is any element in general position. As a consequence, there is a 1--1 correspondence between the nilpotent $G_0^c$-orbits
in $\g_1^c$ and the $G_0^c$-conjugacy classes of carrier algebras.

Let $\h_0^c$ be a fixed Cartan subalgebra of $\g_0^c$. If $\fb^c\leq \g^c$ is
a subalgebra normalised by $\h_0^c$, then $\fb^c$ is called
$\h_0^c$-regular. Since we are interested in listing the carrier algebras
up to $G_0^c$-conjugacy, and all Cartan subalgebras of $\g_0^c$ are
$G_0^c$-conjugate, we can restrict attention to the $\h_0^c$-regular carrier
algebras. In the following we assume, for simplicity, that $\h_0^c$ is
also a Cartan subalgebra of $\g^c$ (as in Section \ref{secsl2comp}); note that this holds in our main example (Example \ref{exa:1}). The general case does not pose extra
difficulties, but is just a bit more cumbersome to describe.

Let $W_0^c = N_{G_0^c}(\h_0^c)/Z_{G_0^c}(\h_0^c)$
be the Weyl group of $\g_0^c$ with
respect to $\h_0^c$; this group is isomorphic to the Weyl group of the
root system $\Phi_0$ of $\g_0^c$ with respect to $\h_0^c$. Since $\h_0^c$ is a
Cartan subalgebra $\h^c$ of $\g^c$, it follows that $\Phi_0$ is a subsystem of the root system $\Phi$ of $\g^c$ with respect to $\h^c$. Moreover, $W_0^c$
is in a natural way a subgroup of the Weyl group of $\Phi$; in particular, it
acts on $\Phi$. For an $\h_0^c$-regular subalgebra $\fb^c\leq \g^c$,
we denote by $\Psi(\fb^c)\subseteq \Phi$ the subset of roots $\alpha\in \Phi$ with $\g_\alpha^c\subseteq \fb^c$; here, as usual, $\g_\alpha^c$ denotes the root space of $\g$ corresponding to the root $\alpha$. We have the following theorem, see \cite[Proposition 4(2)]{vinberg2}.

\begin{theorem}\label{thm:conjcar}
Two $\Z$-graded  $\h_0^c$-regular subalgebras $\fa^c$ and $\fb^c$ are
$G_0^c$-conjugate if and only if $\Psi(\fa^c)$ and $\Psi(\fb^c)$ are
$W_0^c$-conjugate.
\end{theorem}

This reduces the classification of the nilpotent $G_0^c$-orbits in $\g_1^c$
to the classification of root subsystems of $\Phi$, with certain properties,
up to  $W_0^c$-conjugacy; the latter is a finite combinatorial problem.
This approach has been used in a number of publications
(for example, \cite{elashvin})
to classify the nilpotent orbits of a particular $\theta$-representation.
In \cite{gra15,litt8}, this method is  the basis of an
implemented algorithm.

\begin{example}\label{exa:3}
We continue with Examples \ref{exa:1}, \ref{exa:2}, and \ref{exa:2b}. Let $\fs^c
\leq \g^c$ be a carrier algebra which is normalised
by the Cartan subalgebra with basis $\{h_1,\ldots,h_4\}$. Since this is a
Cartan subalgebra of $\g^c$ as well, the root system of $\fs^c$ is a subsystem
of the root system of $\g^c$. It follows that $\fs^c$ can be of type $A_1$, $A_2$, or $A_3$ (or direct sums of those), or $D_4$. It can be shown that a carrier algebra of type $A_i$ is always {\em principal} (see \cite[p.\ 29]{vinberg2}), that is $\fs^c_0$ is
the Cartan subalgebra of $\fs^c$, and so $\fs^c_1$ is spanned by the simple root vectors of $\fs^c$.  Here we consider the problem of listing the carrier algebras of type $A_3$, up to
$G_0^c$-conjugacy. In this case, where the root system of $\g^c$ is not too big,
we can use a brute force approach. Let $\Delta = \{ \alpha_1,\ldots,
\alpha_4\}$ be the simple roots (where we use the ordering according to the
enumeration of the Dynkin diagram in Example \ref{exa:2}). Then $\g_1^c$ is
spanned by the root spaces corresponding to the following roots:
\begin{eqnarray*}
&& \pm \{\alpha_2,\; \alpha_1+\alpha_2,\; \alpha_2+\alpha_3,\; \alpha_2+\alpha_4,\;
 \alpha_1+\alpha_2+\alpha_3,\;\alpha_1+\alpha_2+\alpha_4,\;\\
&&\hspace*{3ex}\alpha_2+\alpha_3+
\alpha_4,\;\alpha_1+\alpha_2+\alpha_3+\alpha_4\}.
\end{eqnarray*}
By a brute force computation, we find 96 three-element subsets each having a Dynkin diagram of type $A_3$. The Weyl group of $\g_0^c$ is generated by the reflections $s_{\alpha_i}$, $i=0,1,3,4$, and another computation shows that it has 6 orbits on these three-sets. It is not difficult
to show that all 6 orbits yield a carrier algebra (for example, see \cite{gra15} for a criterion for this). One such three-set consists of the elements
$\beta_1=\alpha_2$, $\beta_2=\alpha_1+\alpha_2+\alpha_3+\alpha_4$, and
$\beta_3=-\alpha_1-\alpha_1-\alpha_1$. If $x_\beta$ denotes a root vector
corresponding to a root $\beta$, then $x_{\beta_1}+x_{\beta_2}+x_{\beta_3}$
is a representative of the nilpotent orbit corresponding to the carrier
algebra that we found. (In general, for a principal carrier algebra $\fs^c$,
the sum of the basis elements of $\fs^c_1$ is always such a representative.)

From this example it becomes clear that for the higher dimensional cases we
need more efficient methods than just brute force enumeration (cf.\ \cite{gra15}). \exdone
\end{example}

\subsection{The real case}\label{secrealCA}
One can adapt Vinberg's method of carrier algebras to the real case to obtain the nilpotent $G_0$-orbits in $\g_1$; this
has been worked out in \cite{dfg2}. We sketch the two main steps here: constructing the carrier algebras in $\g$ up to $G_0$-conjugacy and, from that, getting the nilpotent $G_0$-orbits in $\g_1$. As in the previous subsection we assume
that a Cartan subalgebra of $\g_0$ is also a Cartan subalgebra of $\g$.
Again, the general case does not pose extra difficulties.

\begin{remark} In some places in \cite{dfg2} rather restrictive assumptions on the group
$G_0$ are posed, namely that $G_0$ is both connected (in the real Euclidean
topology) and of the form $G_0=G_0^c(\R)$.  This assumption affects the algorithm for computing
the real Weyl group in \cite[Section 3]{dfg2} and some of the procedures in \cite[Section 10.2.2]{dfg2} used to obtain elements in $G_0$ belonging to a given split torus.
 However, the main theorems in \cite{dfg2} underpinning the procedure used to classify the real
carrier algebras (as summarised below) are not affected, and, moreover, it is straightforward to verify that these theorems can also be applied to the more general definition of the group $G_0$ given in Section \ref{secsetup}; for this reason we do not repeat the full proofs of these results here.
\end{remark}

\subsubsection{Listing the carrier algebras}
Let $e\in \g_1$ be nonzero nilpotent. Its carrier algebra, $\fs(e)$,
is defined  as in the complex case, except that we choose $\h_z$ to be a
{\em maximally noncompact} Cartan subalgebra of $\fz_{0}(\fa)$.
Since maximally noncompact Cartan subalgebras are unique up to conjugacy
under the adjoint group, if  $e,e'\in \g_1$ are nilpotent and $G_0$-conjugate,
then also their carrier algebras $\fs(e)$ and $\fs(e')$ are $G_0$-conjugate by \cite[Proposition 34]{dfg2}. Let $\h_0$ be a Cartan
subalgebra of $\g_0$. An $\h_0$-regular $\Z$-graded semisimple
subalgebra $\fs\leq \g$ is {\em strongly $\h_0$-regular} if $\h_0$ is a
maximally noncompact
Cartan subalgebra of the reductive Lie algebra $\fn_{0}(\fs)=\{x\in\g_0\mid [x,\fs]\subseteq \fs\}$. Let $\Phi(\h_0)$ denote the root
system of $\g^c$  with respect to $\h_0^c$.
For an $\h_0$-regular subalgebra $\fs\leq \g$ denote by $\Psi(\fs)$ the
set of roots $\alpha\in \Phi(\h_0)$ with $\g_\alpha^c\subseteq \fs^c$. Now we have
the following analogue of Theorem \ref{thm:conjcar}, see
\cite[Proposition 24]{dfg2} for its proof.

\begin{proposition}\label{prop:conjstrong}
Two $\Z$-graded semisimple strongly $\h_0$-regular subalgebras
$\fs$ and $\fs'$ are $G_0$-conjugate if and only if $\Psi(\fs)$ and $\Psi(\fs')$
are conjugate under the real Weyl group $W(\h_0)= N_{G_0}(\h_0)/Z_{G_0}(\h_0)$.
\end{proposition}

We recall from \cite[(7.93)]{knapp} that $W(\h_0)$ is a subgroup of the complex Weyl group of $\g_0^c$ with respect to $\h_0^c$, hence it acts on complex roots.

This leads to the following approach for listing the carrier algebras of
$\g$ up to $G_0$-conjugacy.
Let $\h_0$ be a Cartan subalgebra of $\g_0$. We compute the carrier algebras
of $\g^c$ up to $G_0^c$-conjugacy, using the root system $\Phi(\h_0^c)$ of
$\g^c$ relative to $\h_0^c$. For each such carrier algebra $\fs^c$ we first
compute all root systems $w\cdot \Psi(\fs^c)$, where $w$ runs over the Weyl
group of $\Phi(\h_0^c)$ and then compute the semisimple subalgebras with these
root systems. These subalgebras are all the  $\h_0^c$-regular carrier algebras
in $\g^c$.   We eliminate those that are not
contained in $\g$, and those that are not strongly $\h_0$-regular.
Furthermore, we eliminate copies that are conjugate under the real Weyl group
$W(\h_0)$. Let $\h_0^1,\ldots,\h_0^{t}$ be the Cartan subalgebras of $\g_0$
up to $G_0$-conjugacy; we carry out the outlined procedure for each
$\h_0^i$, and thereby find all carrier algebras in $\g$, up to
$G_0$-conjugacy.

We note that the list of Cartan subalgebras of $\g_0$, up to $G_0$-conjugacy,
can be obtained using algorithms described in \cite{dfg}. (In that reference, conjugacy under the adjoint group is used,
but, for almost all types, that does not make a difference, see
\cite[Theorem 11]{sugiura} or \cite[Theorem 8]{kostantcsa}.)
An algorithm for
computing the real Weyl group corresponding to a given Cartan subalgebra
has been given in \cite{adclo}.
In the case of our main example, however, things are rather straightforward.

\begin{example}
Let the notation be as in Examples \ref{exa:1}, \ref{exa:2},  \ref{exa:2b}, and \ref{exa:3}. Up to conjugacy under $\mathcal{G}=\SL_2(\R)$, the Lie algebra $\sl_2(\R)$ has two Cartan subalgebras: a split one  and a compact one.  Let $\h\subseteq \sl_2(\R)$ be a Cartan subalgebra, with corresponding real
Weyl group $W(\h)= N_{\mathcal{G}}(\h)/Z_{\mathcal{G}}(\h)$. If $\h$ is split then
$W(\h)$ has order 2, and is equal
to the Weyl group of the root system of $\sl_2(\C)$ with respect to $\h^c$. If $\h$ is compact, then $W(\h)$ is trivial. All this can
be verified by a direct calculation, see also  \cite[pp.\ 487 \& 489]{knapp}. It follows that up to $\widetilde{G}_0$-conjugacy, $\tilde \g_0$  has 16
Cartan subalgebras. Thus, $\g_0$ has 16 Cartan subalgebras, up to $G_0$-conjugacy, and they are of the form $\h_0=\h^1\oplus \cdots \oplus \h^4$,
where $\h^i$ is split or compact in the $i$-th direct summand of $\g_0$.
The real Weyl group is
\[W(\h_0) = N_{G_0}(\h_0)/Z_{G_0}(\h_0)= W(\h^1)\times \cdots \times W(\h^4).\]
Note that $\h_0$ is a Cartan subalgebra of $\g$, hence $W(\h_0)$ is a subgroup
of the Weyl group of the root system of $\g^c$ with respect
to the Cartan subalgebra $\h_0^c = \h_0\otimes_\R \C$, see
\cite[(7.93)]{knapp}. In particular, as mentioned above, $W(\h_0)$ acts acts
on that root system.\exdone
\end{example}

\subsubsection{Listing nilpotent orbits}
Unlike the complex case, listing the carrier algebras up to conjugacy does
not immediately yield all nilpotent orbits. As already remarked in the
previous subsection, it can happen that two non-conjugate nilpotent elements
have the same carrier algebra. Let $\fs\subseteq \g$ be a carrier algebra.
We say that a nilpotent $e\in \g_1$ \emph{corresponds} to $\fs$ if there
is a homogeneous $\sl_2$-triple $(h,e,f)$ and a maximally noncompact
Cartan subalgebra in $\fz_0(\fa)$, where $\fa$ is the subalgebra spanned
by $\{h,e,f\}$, such that the carrier algebra constructed like in Section
\ref{seccacomp} is equal to $\fs$. For $e$ to correspond to $\fs$ it is
necessary that $e$ is in general position in $\fs_1$, but this is not enough.

In order to decide whether $e$ corresponds to $\fs$ we use the following
theorem (for whose proof we refer to \cite[Proposition 35]{dfg2}).
In order to formulate it, we need a definition. A subalgebra
$\fb\leq \g$ is said to be an $\R$-split torus, if it is abelian, and for all
$x\in \fb$ we have that $\ad_\g(x)$ is semisimple with eigenvalues in $\R$.
The {\em real rank} of a real reductive Lie algebra is the dimension of a
maximal $\R$-split torus. The next theorem is \cite[Proposition 35(b)]{dfg2}.

\begin{theorem}\label{thm:corrse}
Let $e\in \fs_1$ be in general position, lying in the homogeneous $\sl_2$-triple
$(h,e,f)$ with $h\in \fs_0$, $f\in \fs_{-1}$. Let $\fa$ be the subalgebra
spanned by $\{h,e,f\}$. Then $e$ corresponds to $\fs$ if and only if the real
ranks of $\fz_{\g_0}(\fa)$ and $\fz_{\g_0}(\fs)$ coincide.
\end{theorem}

Given the list of carrier algebras, we need to find the
nilpotent orbits to which each carrier algebra corresponds. Let $\fs$ be such
a carrier algebra. Then we have to find, up to $G_0$-conjugacy, all
nilpotent $e\in \fs_1$ in general position. Fix a basis $\{y_1,\ldots,y_m\}$
of $\fs_1$; it is straightforward to find a polynomial $p\in \R[T_1,\ldots, T_m]$ such that $e=\sum_i t_i y_i$ is in general position in $\fs_1$
if and only if $p(t_1,\ldots,t_m)\neq 0$. This gives an explicit description of the set $\Gamma$ of elements in $\fs_1$ that are
in general position.

The next problem is to reduce the set $\Gamma$, using the action of $G_0$.
We use some strategies for this, that are, however, not always guaranteed to
succeed. It is therefore more appropriate to call them ``methods'' rather
than ``algorithms''. First, we try to find a smaller set $\Gamma_1\subseteq\Gamma$ having two properties:
\begin{items}
\item[$\bullet$] For each $e\in \Gamma$ there exists $g\in G_0$ with $g(e)\in \Gamma_1$.
\item[$\bullet$] There are indices $1\leq i_1<\ldots<i_r\leq m$ such that $\Gamma_1$
consists of all elements $t_{i_1}y_{i_1}+\cdots + t_{i_r} y_{i_r}$ with
$t_{i_1}\cdots t_{i_r}\neq 0$.
\end{items}
We use two constructions of elements of $G_0$ for this; these elements have
to stabilise $\fs_1$, otherwise they are rather difficult to use
for our purpose.
Firstly, if $\fs_0$ contains a nilpotent element $u$, then we consider
$\exp(t\ad_{\g}(u))\in G_0$ for  $t\in \R$. Secondly, if $\fs_0$ has a compact
subalgebra $\fa$ (for example, a compact torus), then we construct the
connected algebraic subgroup $A^c\leq G_0^c$ with Lie algebra
$\ad_{\g^c}(\fa^c)$. Since by a theorem of Chevalley \cite[\S VI.5, Proposition 2]{cheviii} compact Lie groups
are algebraic, the set of real points $A^c(\R)$ is connected in the Euclidean
topology, and therefore contained in $G_0$.
In particular, if $\fs$ is strongly
$\h_0$-regular, then we can apply this construction to compact subalgebras $\fa$ of $\h_0$.

Subsequently, we act again with explicit elements of $G_0$ in order
to find a subset $\Gamma_2\subseteq \Gamma_1$ with the same properties as
$\Gamma_1$, but with the extra condition that the coefficients $t_{i_j}$ lie in
a finite set, preferably in $\{ 1, -1\}$. Good candidates for elements
that help to achieve this come from a torus in $G_0$ that acts diagonally
on $\fs_1$.

Finally from $\Gamma_2$ we eliminate the elements $e$ such that $\fs$ is
not a carrier algebra of $e$; for that we use Theorem \ref{thm:corrse}.
This yields another subset $\Gamma_3\subseteq \Gamma_2$. We then try to show that the elements of
$\Gamma_3$ are pairwise {\em not} $G_0$-conjugate. Section \ref{SNO} gives more details on how we attempt this task. After having done that, $\Gamma_3$ is
the set of nilpotent elements corresponding to $\fs$, up to $G_0$-conjugacy.
In the next subsection we illustrate these
methods by some examples, all relative to our main example.

\subsection{Examples}\label{sec:carexa}  
We now provide a number of explicit examples to illustrate the method described in Section \ref{carrier1}. Throughout, we use the notation of our main example, Examples \ref{exa:1} and \ref{exa:2}, and
write an element of
$\widetilde{G}_0^c(\R)=
(\SL_2(\R))^4$ as
\begin{equation}\label{eqg}
\quad g=\left(\begin{smallmatrix} a_1 & b_1 \\ c_1 & d_1 \end{smallmatrix}\right) \times\cdots
\times \left(\begin{smallmatrix} a_4 & b_4 \\ c_4 & d_4 \end{smallmatrix}\right),
\end{equation}
where each $a_i,b_i,c_i,d_i\in \R$ with $a_id_i-b_ic_i=1$; each $2\times 2$ matrix is given with respect to the chosen basis $(1,0)$ and $(0,1)$ of the natural $\sl_2(\R)$-module $V_2$, see Example \ref{exa:2}. We also  use the fact that the actions of $G_0$ on $\g_1$ and of
$\widetilde{G}_0^c(\R)$ on $V_2^{\otimes 4}=V_2\otimes V_2\otimes V_2\otimes V_2$
are equivalent; we fix a module isomorphism $\sigma \colon V_2^{\otimes 4}
\to \g_1$ and, by abuse of notation, define $\sigma\colon \GL(V_2^{\otimes 4})\to \GL(\g_1)$ by the following commuting diagram:

\begin{displaymath}
\xymatrix{
\widetilde{G}_0^c(\R) \ar[r]^{R^c} \ar[d] & G_0 \ar[d] \\
\GL(V_2^{\otimes 4}) \ar[r]^{\sigma} & \GL(\g_1).
}
\end{displaymath}
We go back and forth: sometimes we use elements of $G_0$ acting on
$\g_1$, sometimes we use elements of $\widetilde{G}_0^c(\R)$. Also, in order
to use the method based on Gr\"obner basis (see Section \ref{sec:gbconj}),
we directly use the action of $\widetilde{G}_0^c(\R)$ on
$V_2^{\otimes 4}$.

\begin{example}
We consider a carrier algebra $\fs$ that, as a Lie algebra, is just $\g$ itself.
It is strongly $\h_0^1$-regular, where $\h_0^1$ is the split Cartan subalgebra
of $\g_0$, spanned by $\{h_1,\ldots,h_4\}$. We have that
\begin{align*}
\fs_0 &= \Span( x_{\alpha_1}, x_{-\alpha_1}, h_1,h_2,h_3,h_4)\\
\fs_1 &= \Span( x_{\alpha_2+\alpha_3}, x_{\alpha_2+\alpha_4},
x_{\alpha_1+\alpha_2+\alpha_3}, x_{\alpha_1+\alpha_2+\alpha_4}, x_{-\alpha_2},
x_{-\alpha_1-\alpha_2}).
\end{align*}
For brevity we also denote the basis elements of $\fs_1$ by $y_1,\ldots,y_6$. We obtain that $\sum_i t_i y_i \in \fs_1$ is in general position if and only if
$p(t_1,\ldots,t_6)\neq 0$, where
$$p=(T_2T_5+T_4T_6)(T_1T_5+T_3T_6)(T_1T_4-T_2T_3).$$
Note that the subalgebra $\fs_0$ contains two root vectors. Thus, we consider the restrictions
of $\exp( s\ad_\g (x_{\alpha_1}))$ and $\exp( s\ad_\g(x_{-\alpha_1}))$ to $\fs_1$.
Relative to the given basis of $\fs_1$, they have matrices
$$A(s)=\left(\begin{smallmatrix} 1 & 0 & 0 & 0 & 0 & 0\\
0 & 1 & 0 & 0 & 0 & 0\\
-s & 0 & 1 & 0 & 0 & 0\\
0 & -s & 0 & 1 & 0 & 0\\
0 & 0 & 0 & 0 & 1 & s\\
0 & 0 & 0 & 0 & 0 & 1
\end{smallmatrix}\right)\quad\text{ and }\quad
B(s)=\left(\begin{smallmatrix} 1 & 0 & -s & 0 & 0 & 0\\
0 & 1 & 0 & -s & 0 & 0\\
0 & 0 & 1 & 0 & 0 & 0\\
0 & 0 & 0 & 1 & 0 & 0\\
0 & 0 & 0 & 0 & 1 & 0\\
0 & 0 & 0 & 0 & s & 1
\end{smallmatrix}\right).$$
(We use the column convention: the first column of $A(s)$ gives the
image of $\exp( s\ad_\g(x_{\alpha_1}))$ on $y_1$ and so on.)
By inspecting the polynomial $p$, we see that an element in general position cannot have
both $t_1=t_3=0$. If $t_1=0$, then after acting with $B(1)$ we get a new
element with $t_1\neq 0$, so we may assume $t_1\neq 0$. By acting with
$A(t_3/t_1)$ we obtain a new element with $t_3=0$, so we may assume $t_3=0$.
Using $p$, we deduce that $t_4t_5\ne 0$. By acting
with $B(t_2/t_4)$ we construct a new element with $t_2=0$. In conclusion, our set $\Gamma_1$
consists of $t_1y_1+t_4y_4+t_5y_5+t_6y_6$ with all $t_1t_4t_5t_6\neq 0$.

In order to obtain the set $\Gamma_2$ it is useful to transform $\Gamma_1$
to a subset of $V_2^{\otimes 4}$. It turns out that this transformation maps
$y_1 \mapsto -(-++-)$, $y_4\mapsto (++-+)$, $y_5\mapsto (+---)$, $y_6\mapsto
(----)$,  so the elements of $\Gamma_1$ are
$$u=u_1 (-++-)+u_2(++-+)+u_3(+---)+u_4(----),$$
with all $u_i\neq 0$. Now take a $g$ as in \eqref{eqg}, with $b_i=c_i=0$ and
$d_i=a_i^{-1}$ for all $i$. Then
\begin{align*}
g\cdot (-++-) &= a_1^{-1}a_2a_3a_4^{-1} (-++-)\\
g\cdot (++-+) &= a_1a_2a_3^{-1}a_4 (++-+)\\
g\cdot (+---) &= a_1a_2^{-1}a_3^{-1}a_4^{-1} (+---)\\
g\cdot (----) &= a_1^{-1}a_2^{-1}a_3^{-1}a_4^{-1} (----),
\end{align*}
and we can easily find such a $g$ such that $g\cdot u$ has $u_1=1$, so we
may assume that $u_1=1$. Now we use a different $g$, one which does not change the first coefficient $u_1=1$, that is, we need $g$ with $a_2a_3=a_1a_4$. Acting with such a $g$ means that the third coefficient $u_3$ is multiplied by $a_4^{-2}$, so  we may assume that
 $u_3=\pm 1$. Continuing like that, we get that $u_1=1$ and $u_2,u_3,u_4\in\{\pm1\}$, thus the elements in our set $\Gamma_2$ are
$$(-++-)\pm (++-+)\pm (+---)\pm (----).$$

It turns out that all these elements have $\fs$ as carrier algebra, so $\Gamma_3=\Gamma_2$.
Moreover, Gr\"obner basis computations
(see Section \ref{sec:gbconj}) show that all elements of $\Gamma_2$
are not $\widetilde{G}_0^c(\R)$-conjugate. We conclude that this carrier
algebra corresponds to 8 real nilpotent orbits.\exdone
\end{example}

\begin{example}
Here we consider a carrier algebra $\fs$ that, as Lie algebra, is isomorphic
to $\sl_3(\R)$. It is strongly $\h_0^1$-regular, as in the previous example,
and here
$$\fs_1 = \Span( x_{\alpha_1+\alpha_2+\alpha_3+\alpha_4},x_{-\alpha_1-\alpha_2-\alpha_3}).$$
We also denote these basis elements by $y_1,y_2$. An element $t_1y_1+t_2y_2$
is in general position if and only if $t_1t_2\neq 0$, which immediately gives us the set $\Gamma_1$. As in the previous example, we transform the elements to
$V_2^{\otimes 4}$; we obtain that the elements of $\Gamma_1$ are
$u_1(++--)+u_2(---+)$.
As before, we construct $\Gamma_2=\Gamma_3$, whose elements are
$(++--)\pm (---+)$. However, in this case, by considering the polynomial
equations equivalent to the $\widetilde{G}_0^c(\R)$-conjugacy of these
elements, and solving them, we find the element
$$g = \left(\begin{smallmatrix} 0 & 1 \\ -1 & 0 \end{smallmatrix}\right)\times
\left(\begin{smallmatrix} 0 & 1 \\ -1 & 0 \end{smallmatrix}\right)\times
\left(\begin{smallmatrix} 1 & 0 \\ 0 & 1 \end{smallmatrix}\right)\times
\left(\begin{smallmatrix} 0 & 1 \\ -1 & 0 \end{smallmatrix}\right),$$
which maps $(++--)-(---+)$ to $(++--)+(---+)$. So the two elements of $\Gamma_3$
are conjugate, and we conclude that this carrier algebra corresponds to one
nilpotent orbit.\exdone
\end{example}

\begin{example}
In this example, $\fs$ is a carrier algebra that, as Lie algebra, is isomorphic
to $\mathfrak{su}(1,2)$; it is a non-split real form of $\sl_3(\R)$.
It is strongly $\h_0^{14}$-regular, where $\h_0^{14}$ is spanned by
$$\{h_4,~ x_{\alpha_1}-x_{-\alpha_1},~x_{\alpha_3}-x_{-\alpha_3},x_{\alpha_0}-x_{-\alpha_0}\},$$
so it has compact dimension 3, and non-compact dimension 1. We have
\begin{eqnarray*}\fs_1&=&\Span( x_{\alpha_2+\alpha_4}-x_{\alpha_1+\alpha_2+\alpha_3+\alpha_4}+
x_{-\alpha_1-\alpha_2}+x_{-\alpha_2-\alpha_3},\\&&\hspace*{1ex} x_{\alpha_1+\alpha_2+\alpha_4}+
x_{\alpha_2+\alpha_3+\alpha_4}-x_{-\alpha_2}+x_{-\alpha_1-\alpha_2-\alpha_3}),
\end{eqnarray*}
and denote these basis elements by $y_1,y_2$. This time  $t_1y_1+t_2y_2$
is in general position if and only if $t_1^2+t_2^2\neq 0$. In order to find the
set $\Gamma_1$, we consider $x=x_{\alpha_1}-x_{-\alpha_1}\in \h_0^{14}$, which stabilises $\fs_1$. We let $A^c\subseteq G_0^c$ be the connected
algebraic group whose Lie algebra is spanned by $\ad_\g(x)$. This group is compact,
and therefore its set of real points $A^c(\R)$ is contained in $G_0$.
We have that $A^c$ is a 1-dimensional torus, and we can construct an
isomorphism $\C^* \to A^c$, $t\mapsto a(t)$. The matrix of the restriction
of $a(t)$ to $\fs_1$, with respect to the given basis, is
$$\left(\begin{smallmatrix} \frac{1}{2}(t+t^{-1}) & \frac{1}{2}\imath (-t+t^{-1})\\[1ex]
\frac{1}{2}\imath
(t-t^{-1}) & \frac{1}{2}(t+t^{-1}) \end{smallmatrix}\right).$$
This matrix has coefficients in $\R$ if and only if $t=x+\imath y$ with
$x,y\in \R$ and $x^2+y^2=1$. With these substitutions, the restriction of $a(t)$ to $\fs_1$ has the form
$$\left(\begin{smallmatrix} x & y \\ -y & x \end{smallmatrix}\right).$$
Using such elements, it is easy to see that each nonzero element of
$\fs_1$ is conjugate to a multiple of $y_1$. Under the identification
of $\g_1$ and $V_2^{\otimes 4}$, the element $y_1$ maps to
$$u=-(-+-+)-(++--)-(+--+)+(----).$$
By acting with a diagonal element, like in the first example, it is
straightforward to see that $u$ is conjugate to any nonzero scalar multiple
of $u$: it suffices to set $a_i=1$, except $a_3=\lambda^{-1}$ to get
an element of $\widetilde{G}_0^c(\R)$ that multiplies $u$ by $\lambda$.
It follows that
also in this case $\fs$ corresponds to one nilpotent orbit.
Over $\C$ it is conjugate to the orbit of the previous example. Over $\R$
they cannot be conjugate as the carrier algebras are not isomorphic.\exdone
\end{example}

\section{Separating real nilpotent orbits}\label{SNO}
\noindent We consider the problem to decide whether two given nonzero nilpotent
$e,e'\in \g_1$ are conjugate under $G_0$. We describe a few invariants
that allow in some cases to prove that $e$ and $e'$ are \emph{not} conjugate, namely, if the invariant in question has different values for $e$ and $e'$. We also briefly discuss a method based on polynomial equations and
Gr\"obner bases, by which it is also possible, on some occasions, to find
a conjugating element, thereby showing that $e$ and $e'$ are conjugate.

\subsection{$\pmb{\alpha}$-labels}

Since $e,e'\in \g_1$,  they also lie in $\g^c$. Clearly, if $e$ and $e'$ are not
$G^c$-conjugate, then they cannot be $G_0$-conjugate either.
The $\alpha$-labels of a nilpotent element $e\in \g^c$ identify the
$G^c$-orbit of $e$; they are the labels of the weighted Dynkin diagram of
the orbits of $e$ (see \cite{colmcgov}). We assume that $e$ lies in an
$\sl_2$-triple $(h,e,f)$ with $h$ in a given Cartan subalgebra $\h^c$ of $\g^c$.
(Note that this holds for the nilpotent elements that are found by our
classification methods.) Let $\Phi$ denote the root system
of $\g^c$ with respect to $\h^c$. Let $\Delta = \{ \alpha_1,\ldots,\alpha_\ell\}$
be a fixed set of simple roots, and let $W$ denote the Weyl group.
Then $W$ also acts on $\h^c$, and there is a unique $h'\in \h^c$ which is
$W$-conjugate to $h$, and such that $\alpha_i(h')\geq 0$ for $1\leq i\leq \ell$.
Then $\alpha_i(h')\in \{0,1,2\}$ for all $i$, see \cite[Section 3.5]{colmcgov}; these are the $\alpha$-labels
of $e$.

\subsection{$\pmb{\beta}$-labels}\label{sec:betalab}

Let $G$ be the connected Lie group with Lie algebra $\ad_\g (\g)$. Here we
suppose that $G_0$ is connected as well. Then $G_0\subseteq G$, so
if $e$ and $e'$ are not $G$-conjugate, then they cannot be $G_0$-conjugate either.
The $\beta$-labels identify the $G$-orbit of $e$, using the
Kostant-Sekiguchi correspondence (see \cite{colmcgov} to which we refer also for the background of the rest of this subsection). In order to describe
how to compute the $\beta$-labels, we assume that $e$ lies in a real
Cayley triple $(h,e,f)$, see Section \ref{sec:realsl2}.  (Note that the
nilpotent elements found by the method of that section do lie in
such a triple.) Let $\g=\fk\oplus \fp$ be the fixed Cartan decomposition of
$\g$, and let $\g^c = \fk^c\oplus \fp^c$ be its complexification.
First we compute the {\em Cayley transform} $(h_0,e_0,f_0)$ of the triple
$(h,e,f)$, where $h_0=\imath (e-f)$, $e_0= \tfrac{1}{2}(-\imath e
-\imath f +h)$, $f_0=\tfrac{1}{2}(\imath e+\imath f +h)$. Then $(h_0,e_0,f_0)$
is a homogeneous $\sl_2$-triple in the $\Z/2\Z$-graded Lie algebra $\g^c$.
Let $K^c$ denote the connected algebraic subgroup of $G^c$ with Lie algebra
$\ad_{\g^c}(\fk^c)$. Then the Cayley transform induces a bijection between
the nilpotent $G$-orbits in $\g$ and the nilpotent $K^c$-orbits in $\fp^c$.
Assume that $h_0$ lies in a fixed Cartan subalgebra $\h_0^c$ of $\fk^c$.
The latter algebra is reductive and we
let $\Phi_k$ denote the root system of $\fk^c$, with respect to $\h_0^c$, with
Weyl group $W_k$. Let $\{\beta_1,\ldots,\beta_m\}$ be a fixed set of simple
roots of $\Phi_k$. Then $h_0$ is $W_k$-conjugate to a unique $h_0'$ such that
$\beta_i(h_0')\geq 0$ for $1\leq i\leq m$. Now these numbers are the
$\beta$-labels of $e$. Using Lemma~\ref{lem:sl2C}
(with $K^c$ in place of $G_0^c$, and so on) it
is seen that they uniquely determine the $K^c$-orbit of $e_0$.
By the Kostant-Sekiguchi correspondence, they uniquely determine the
$G$-orbit of $e$ as well.

In the example under consideration (Example \ref{exa:1} and following), we identify the simple roots in $\Phi_k$ with the ones in $D_4$ via $\beta_1=-(\alpha_1+2\,\alpha_2+\alpha_3+\alpha_4)$, $\beta_2=\alpha_1$, $\beta_3=\alpha_4$, and $\beta_4=\alpha_3.$

\subsection{$\pmb{\gamma}$-labels}

Since $e,e'\in \g_1$,  they also lie in $\g^c$. Clearly, if $e$ and $e'$ are not
$G_0^c$-conjugate, then they cannot be $G_0$-conjugate either. We assume
that $e$ lies in a homogeneous $\sl_2$-triple $(h,e,f)$. Furthermore, assume
that $h$ lies in a fixed Cartan subalgebra $\h_0^c$ of $\g_0^c$. Let
$\Phi_0$ denote the root system of $\g_0^c$ with respect to $\h_0^c$,
with Weyl group $W_0$ and fixed set of simple roots $\{\gamma_1,\ldots,
\gamma_n\}$. Then $h$ is $W_0$-conjugate to a unique $h'\in \h_0^c$ with
$\gamma_i(h')\geq 0$ for all $i$. These are the $\gamma$-labels of $e$.
By Lemma \ref{lem:sl2C}, they uniquely identify the $G_0^c$-orbit of $e$.\par
In the example under consideration (Example 1 and following), the sets of $\beta$- and $\gamma$-labels coincide because $\mathfrak{k}^c$ and $\mathfrak{g}_0^c$ are conjugate in $\g^c$.

\begin{remark}\label{rem:delta}
As mentioned earlier, the $\gamma$- and $\beta$-labels do not provide a complete
classification of the real nilpotent orbits. However, orbits with the same $\gamma$- and $\beta$-labels can be distinguished by using tensor classifiers (see the next section). Here, for convenience, we enumerate the different orbits with the same  $\gamma$- and $\beta$-labels simply by $\delta^{(1)}$, $\delta^{(2)}$, \ldots, see Tables \ref{tabletot} and \ref{tablelable}. These "$\delta$-labels" are just a notational tool for differentiating between orbits and have no other intrinsic algebraic or geometric meaning. They are used to distinguish between nilpotent elements with the same $\gamma$ and $\beta$-labels but which differ by the signatures of tensor classifiers, see Section \ref{tclass} and, in particular, Example \ref{exlong}.
\end{remark}

\subsection{Tensor classifiers}\label{tclass}In this section we describe the concept of tensor classifier, which has proven to provide important invariants for distinguishing real nilpotent orbits. Unless otherwise noted, the notation in this section is independent of the rest.

We first recall the definition of a $(p,q)$-tensor of a reductive real Lie group $G$, see  \cite{Chambadal}. Let $D\colon G\to {\rm GL}(V_n)$ be an $n$-dimension representation of $G$, let $\{e_I\mid I=1,\ldots,n\}$ be a basis of $V_n$, and denote by  $V_n^\vee=\{f\colon V_n\to\R \mid f\text{ linear}\}$  the dual space with basis $\{e^I\mid I=1,\ldots,n\}$, such that $e^I(e_J)=\delta_{I,J}$ is the Kronecker-delta. With respect to the chosen bases,   the image $D(g)$ of $g\in G$ corresponds to an $n\times n$ matrix ${\bf D}(g)$ whose  entries ${\bf D}(g)^J{}_I$ are defined by
 \begin{align*}
 &D(g)(e_I)=\sum\nolimits_{J=1}^n {\bf D}(g)^J{}_I\,e_J\,.
 \end{align*}
Similarly, each $v\in V_n$ is described by components $v^I\in\R$ with  $v=\sum_{I=1}^n v^I \,e_I$ If $g\in G$ is given, then the components $v^{\prime I}$ of $v'=D(g)(v)$ are obtained by the multiplication  of ${\bf D}(g)$ with the column vector $(v^I)$, namely
 \begin{equation}
 v^{\prime I}=\sum\nolimits_{J=1}^n{\bf D}(g)^I{}_J\,v^J\,.\label{vpv}
 \end{equation}
 In tensor calculus it is convenient to use \emph{Einstein's summation convention}: whenever an index is repeated in an upper and in a lower position, summation over that index is understood. This yields the following short-hand notation for \eqref{vpv}:
 \begin{equation*}
 v^{\prime I}={\bf D}(g)^I{}_J\,v^J.
 \end{equation*}
Note that $D$ induces a representation $D^\vee$ of $G$ acting on $V_n^\vee$; the $n\times n$ matrix  ${\bf D}^\vee(g)$ representing $D^\vee(g)$ satisfies ${\bf D}^\vee(g)=({\bf D}(g)^{-1})^\intercal$. For positive integers $p,q\geq 0$ the $(p,q)$-\emph{tensor representation} $D^{(p,q)}$ of $G$  acts on  $V^{(p,q)}=(V_n)^{\otimes p}\otimes (V_n^\vee)^{\otimes q}$ and is defined by
\begin{equation*}
D^{(p,q)}=\underbrace{D\otimes \dots\otimes D}_p\otimes \underbrace{D^\vee\otimes \dots\otimes D^\vee}_q\,:\quad G\to {\rm GL}(V^{(p,q)}).
\end{equation*}
We fix $\{e_{I_1}\otimes \dots e_{I_p}\otimes e^{J_1}\otimes \dots e^{J_q}\mid \text{ each }I_i,J_i=1,\ldots,n\}$ as a basis of $V^{(p,q)}$. Each $T\in V^{(p,q)}$ is called a \emph{$(p,q)$-tensor}, and uniquely determined by its components
\begin{equation*}
T=T^{I_1\dots I_p}{}_{J_1\dots J_q}\cdot (e_{I_1}\otimes \dots e_{I_p}\otimes e^{J_1}\otimes \dots e^{J_q})\,,
\end{equation*}
where, according to our convention, summation over the indices $I_1,\dots,\,I_p,\,J_1,\dots, J_q$ is understood. As usual, we identify a tensor with its components relative to our chosen basis. The action of $g\in G$ on a $(p,q)$-tensor $T$ is defined by the relation between the components of $T$ and those of $T'=D^{(p,q)}(g)(T)$ with respect to the same basis, generalising (\ref{vpv}) to
\begin{equation}
T^{\prime\,I_1\dots I_p}{}_{J_1\dots J_q}={\bf D}(g)^{I_1}{}_{K_1}\dots {\bf D}(g)^{I_p}{}_{K_p}{\bf D}^\vee(g)_{J_1}{}^{L_1}\dots {\bf D}^\vee(g)_{J_q}{}^{L_q}\,T^{K_1\dots K_p}{}_{L_1\dots L_q}\,.\label{TTra}
\end{equation}
A tensor representation $D^{(p,q)}$ is in general reducible. In the following we consider \emph{symmetric}  $(2,0)$  (or \emph{contravariant}) tensors of the form $T^{IJ}=T^{JI}$, and symmetric  $(0,2)$ (or \emph{covariant}) tensors of the form $T_{IJ}=T_{JI}$. These span $G$-invariant vector spaces $V^{(2,0)}_s$ and $V^{(0,2)}_s$ of dimension $n(n+1)/2$, and therefore yield representations $D^{(2,0)}_s$ and $D^{(0,2)}_s$ of $G$. In general, these representations may further be reduced to irreducible ones:
\begin{eqnarray*}
D^{(2,0)}_s&=& D_{(1)}\oplus  D_{(2)}\oplus \dots \oplus D_{(m)}\\
D^{(0,2)}_s&=& D^\vee_{(1)}\oplus  D^\vee_{(1)}\oplus \dots \oplus D^\vee_{(m)},
\end{eqnarray*}
with corresponding vector space decompositions
\begin{equation*}
V^{(2,0)}_s=V_{(1)}\oplus  V_{(2)}\oplus \dots \oplus V_{(m)}\quad \text{ and }\quad V^{(0,2)}_s=V^\vee_{(1)}\oplus  V^\vee_{(2)}\oplus \dots  \oplus V^\vee_{(m)}\,.
\end{equation*}

For any $g\in G$, generic matrices ${\bf T}\in V_{(\ell)}$ and ${\bf T}^\vee\in V^\vee_{(\ell)}$ (with $\ell=1,\dots, m$) can be transformed according to the rule (\ref{TTra}), which amounts to the following congruence transformations:
\begin{equation}
{\bf T}\longrightarrow {\bf T}'={\bf D}(g)\,{\bf T}\,{\bf D}(g)^\intercal\quad\text{and}\quad{\bf T}\longrightarrow {\bf T}^{\vee\,\,\prime}={\bf D}^\vee(g)\,{\bf T}^\vee\,{\bf D}^\vee(g)^\intercal,\label{symtras}
\end{equation}
where the products in the right hand sides of these equations are formed
by the ordinary matrix multiplication. In turn, this
defines the representations $D_{(\ell)}$ and $D^\vee_{(\ell)}$ of $G$.

 \begin{definition}\label{def:sign} The \emph{signature}  of a square symmetric real matrix ${\bf M}$ is ${\sign}({\bf M})=(n_+,n_-)$ where $n_+$ and $n_-$ are the number of eigenvalues of ${\bf M}$ which are positive and negative, respectively.
\end{definition}
The idea behind the notion of \emph{tensor classifiers} is that the signature of a symmetric  $(2,0)$  or $(0,2)$ tensor is  $G$-invariant, being preserved by the congruence transformations \eqref{symtras}; this follows from Sylvester's Law \cite[Theorem 21.5 p.\ 376]{Chambadal}. Since Sylvester's Law does \emph{not} apply to congruence transformations by complex matrices, the signature of a tensor classifier  is an invariant which might allow us to distinguish  different $G$-orbits within a $G^c$-orbit. For this it is important that the relevant symmetric (covariant or contravariant) tensors are restricted to subspaces $V_{(\ell)}$ supporting irreducible representations of $G$.

In the context of real nilpotent orbits of symmetric spaces, tensor classifiers have first been applied to  the problem of separating the $G_0$-orbits in $\mathfrak{g}_1$, see \cite{Fre:2011uy} (and \cite{Chemissany:2012nb} for a later application). Let $\rho$ be the representation defining the action of $G_0$ on $\mathfrak{g}_1$; for any $v\in \mathfrak{g}_1$ we can consider tensors $v^{\otimes 2}=v\otimes v$, $v^{\otimes 3}=v\otimes v\otimes v$, etc., acted on by $\rho^{\otimes 2}=\rho\otimes \rho$, $\rho^{\otimes 3}=\rho\otimes \rho\otimes \rho$, etc. Each of these tensor products $\rho^{\otimes k}$, acting on $v^{\otimes k} $, is decomposed into irreducible representations $\rho_{(\ell)}$ of $G_0$. For some of these irreducible representations it may happen that the restriction of $v^{\otimes k} $ to the corresponding underlying space is described by a symmetric tensor $T^{IJ}(v)$ or $T_{IJ}(v)$, homogeneous functions of degree $k$ of the components of $v$ (with respect to a chosen basis). The signature of  $T^{IJ}(v)$ or $T_{IJ}(v)$ is then a $G_0$-invariant associated with $\rho$, that is, $\mbox{sign}(T^{IJ}(v))=\mbox{sign}(T^{IJ}(\rho(g)(v)))$ for all $g\in G_0$. We illustrate this construction with an example.

\begin{example}\label{exlong}
We use the notation of Example \ref{exa:2}. Let $\{e^{(1)}_{\alpha_1}\otimes e^{(2)}_{\alpha_2}\otimes e^{(3)}_{\alpha_3}\otimes e^{(4)}_{\alpha_4}\mid \alpha_1,\ldots,\alpha_4\in\{1,2\}\}$ be a basis of $V_2^{\otimes 4}$, where $e^{(i)}_{1}$ and $e^{(i)}_2$ are two generators of the $i$-th factor $V_2$. A generic element $v\in V_2^{\otimes 4}$ can then be written as $v=v^{\alpha_1\alpha_2\alpha_3\alpha_4}\cdot(e^{(1)}_{\alpha_1}\otimes e^{(2)}_{\alpha_2}\otimes e^{(3)}_{\alpha_3}\otimes e^{(4)}_{\alpha_4})$ where, again, Einstein's summation convention is used. Given $v\in \mathfrak{g}_1$, we can represent it in terms of the components $v^{\alpha_1\alpha_2\alpha_3\alpha_4}$ of the corresponding element in $V_2^{\otimes 4}$. For each ${\rm SL}_2(\mathbb{R})$-factor in $G_0=(\SL_2(\R))^4$, the antisymmetric matrix
\begin{equation*}
\epsilon=(\epsilon_{\alpha\beta})=\left(\begin{smallmatrix}0 & 1\cr -1 & 0\end{smallmatrix}\right)
\end{equation*}
 is an invariant tensor: if ${\bf g}=(g_\alpha{}^\beta)$ is the $2\times 2$ matrix representing $g\in {\rm SL}_2(\mathbb{R})$ in the fundamental representation\footnote{We denote the ${\rm SL}_2(\mathbb{R})$ representations by their dimension in boldface, so that, for instance, ${\bf 3}$ denotes the adjoint representation.} ${\bf 2}$, then according to \eqref{TTra} (and using $\det({\bf g})=1$) we have
 \begin{equation}
 \epsilon_{\alpha\beta}\to \epsilon'_{\alpha\beta}=g_\alpha{}^\gamma g_\beta{}^\delta\,\epsilon_{\gamma\delta}=\epsilon_{\alpha\beta}.\label{epsinv}
 \end{equation}
Given $v\in V_2^{\otimes 4}$, we can consider the tensor $v\otimes v$, whose components are  the products $v^{\alpha_1\alpha_2\alpha_3\alpha_4}\,v^{\beta_1\beta_2\beta_3\beta_4}$, and construct symmetric contravariant tensors from it. The tensor $v\otimes v$ is acted on by the representation $\rho\otimes \rho$ of $G_0$, where $\rho={\bf (2,2,2,2)}$; we have the following decomposition into irreducible representations:
\begin{equation}
\rho\otimes \rho = {\bf (3,3,1,1)}\oplus{\bf (3,1,3,1)}\oplus {\bf (3,1,1,3)}\oplus {\bf (1,3,1,3)}\oplus {\bf (1,1,3,3)}\oplus {\bf (1,3,3,1)}+\dots\label{3311}
\end{equation}

We  restrict attention to the six irreducible representations listed above, and associate them with a symmetric contravariant $G_0$-tensor. For example, consider ${\bf (3,3,1,1)}$. The restriction of $v\otimes v$ to its underlying space has the following components:
\begin{equation}
\mathbb{T}_{(1,2)}^{\alpha_1,\alpha_2;\,\beta_1\beta_2}(v)=
\mathcal{P}^{\alpha_1,\beta_1}{}_{\gamma_1,\delta_1}\mathcal{P}^{\alpha_2,\beta_2}{}_{\gamma_2,\delta_2}\,
v^{\gamma_1\gamma_2\alpha_3\alpha_4}\,v^{\delta_1\delta_2\beta_3\beta_4}\epsilon_{\alpha_3\beta_3}\epsilon_{\alpha_4\beta_4}\,,\label{Tvv}
\end{equation}
where $\mathcal{P}^{\alpha,\beta}{}_{\gamma,\delta}$ symmetrizes  a couple of ${\rm SL}_2(\mathbb{R})$-indices in a tensor, that is,
\begin{equation*}
\mathcal{P}^{\alpha,\beta}{}_{\gamma,\delta}=\frac{1}{2}\,\left(\delta_\gamma^\alpha \delta_\delta^\beta+\delta_\delta^\alpha \delta_\gamma^\beta\right)
\end{equation*}
projects components of vectors transforming in the ${\bf 2}\otimes {\bf 2}$ into components of vectors transforming into the ${\bf 3}$ of ${\rm SL}_2(\mathbb{R})$. Note that the indices $\alpha_3,\,\beta_3$ and $\alpha_4,\,\beta_4$ of the two $v$-symbols on the right hand side of (\ref{Tvv}) are summed over (or \emph{contract}) corresponding indices of the $\epsilon$-tensor. Thus, if $v$ is transformed by means of a transformation $g=(g_1,\dots,g_4)\in G_0$, the $g_3,\,g_4$ matrices end up acting on the two $\epsilon$-tensors and thus, by virtue of \eqref{epsinv}, they disappear in the expression of $\mathbb{T}_{(1,2)}^{\alpha_1,\alpha_2;\,\beta_1\beta_2}(v')$ evaluated on the transformed vector $v'=\rho(g)(v)$. The contraction of the last two indices of the two $v$-components with $\epsilon$-tensors therefore makes the resulting tensor $\mathbb{T}_{(1,2)}$ a singlet with respect to the last two ${\rm SL}_2(\mathbb{R})$-groups in $G_0$.

If we define the composite index $I=(\alpha_1\alpha_2)=((1,1),\,(1,2),\,(2,1),\,(2,2))$, the tensor $\mathbb{T}_{(1,2)}$ can be written as a $4\times 4$ matrix with entries $\mathbb{T}_{(1,2)}^{I,J}$. This is a symmetric tensor quadratic in the components of $v$, and, by the property mentioned above, its signature $(n_+,\,n_-)$ is  a $G_0$-invariant associated with the $\mathfrak{g}_1$-element $v$.  For each of the six representations in  (\ref{3311}), we can construct a matrix  $\mathbb{T}_{(i,j)}$, where the pair $(i,j)\in\{1,2,3,4\}^2$ indicates  the couple of ${\rm SL}_2(\mathbb{R})$-factors in $G_0$ with respect to which the tensor is not a singlet.
As an example, let us show how the signatures of these matrices discriminate between $G_0$-orbits (distinguished by the $\beta$-labels) within a same $G_0^c$-orbit (represented by a given $\gamma$-label), see Table \ref{tabletot}. As $G_0^c$-orbit we choose the one labelled by $\gamma^{(6;1)}$, and we obtain:
 \begin{align*}
(\gamma^{(6;1)},\,\beta^{(6;1)})&:\quad\,\,
\mbox{sign}(T_{(1,2)}(v))=\mbox{sign}(T_{(1,3)}(v))=\mbox{sign}(T_{(1,4)}(v))=(0,\,2),\\
(\gamma^{(6;1)},\,\beta^{(6;2)})&:\quad\,\,
\mbox{sign}(T_{(1,2)}(v))=(0,\,2)\quad\text{and}\quad\mbox{sign}(T_{(1,3)}(v))=\mbox{sign}(T_{(1,4)}(v))=(2,\,0),\\
(\gamma^{(6;1)},\,\beta^{(6;3)})&:\quad\,\,
\mbox{sign}(T_{(1,4)}(v))=(0,\,2)\quad\text{and}\quad\mbox{sign}(T_{(1,2)}(v))=\mbox{sign}(T_{(1,3)}(v))=(2,\,0),\\
(\gamma^{(6;1)},\,\beta^{(6;4)})&:\quad\,\,
\mbox{sign}(T_{(1,3)}(v))=(2,\,0)\quad\text{and}\quad\mbox{sign}(T_{(1,2)}(v))=\mbox{sign}(T_{(1,4)}(v))=(2,\,0),\\
(\gamma^{(6;1)},\,\beta^{(6;5)})&:\quad\,\,
\mbox{sign}(T_{(1,2)}(v))=\mbox{sign}(T_{(1,3)}(v))=\mbox{sign}(T_{(1,4)}(v))=(1,\,1).
\end{align*}

Tensor classifiers are also useful to distinguish between orbits with the same $\gamma$- and the same $\beta$-labels:
\begin{align*}
(\gamma^{(6;5)},\,\beta^{(6;5)},\,\delta^{(1)})&:\quad\,\,\,\mbox{sign}(T_{(3,4)}(v))=(0,\,1)\quad\text{and}\quad\mbox{sign}(T_{(2,4)}(v))=(0,\,1),\\
(\gamma^{(6;5)},\,\beta^{(6;5)},\,\delta^{(2)})&:\quad\,\,\,\mbox{sign}(T_{(3,4)}(v))=(0,\,1)\quad\text{and}\quad\mbox{sign}(T_{(2,4)}(v))=(1,\,0),\\
(\gamma^{(6;5)},\,\beta^{(6;5)},\,\delta^{(3)})&:\quad\,\,\,\mbox{sign}(T_{(3,4)}(v))=(1,\,0)\quad\text{and}\quad\mbox{sign}(T_{(2,4)}(v))=(0,\,1),\\
(\gamma^{(6;5)},\,\beta^{(6;5)},\,\delta^{(4)})&:\quad\,\,\,\mbox{sign}(T_{(3,4)}(v))=(1,\,0)\quad\text{and}\quad\mbox{sign}(T_{(2,4)}(v))=(1,\,0).
\end{align*}
One can construct symmetric contravariant tensors which are quartic in the components of $v$; we illustrate just one instance of such tensor classifiers. To this end we need to introduce one more ${\rm SL}_2(\mathbb{R})$-invariant tensor. Let  $\{S_1,S_2,S_3\}$ be a basis of $\mathfrak{sl}_2(\C)$, given by  Pauli matrices
\begin{equation*}
S_1=\left(\begin{smallmatrix}0 & 1 \cr 1 & 0\end{smallmatrix}\right),\quad S_2=\left(\begin{smallmatrix}0 & -i \cr i & 0\end{smallmatrix}\right),\quad S_3=\left(\begin{smallmatrix}1 & 0 \cr 0 & -1\end{smallmatrix}\right).
\end{equation*}
Let $S_{x\,\alpha}{}^\beta$ be the entry of $S_x$ in row $\alpha$ and column $\beta$. Seen as the components of a tensor which is transformed in the representation  ${\bf 3}\otimes {\bf 2}\otimes {\bf 2}$, it turns out that $S_{x\,\alpha}{}^\beta$ is an invariant tensor: if $g\in {\rm SL}_2(\mathbb{R})$, then
\begin{equation*}
S_{x\,\alpha}{}^\beta\to g_x{}^y\,g_\alpha{}^\gamma\,g^{-1}{}_\delta{}^\beta \,S_{y\,\gamma}{}^\delta=S_{x\,\alpha}{}^\beta,
\end{equation*}
which follows  from the property $gS_x g^{-1}=g^{-1}{}_x{}^y\,S_y$, where $g_x{}^y$ is  the $3\times 3$ matrix representing the element $g$ in the adjoint representation ${\bf 3}$. Clearly, by (\ref{epsinv}), also the tensor $S_{x\,\alpha\beta}=S_{x\,\alpha}{}^{\gamma}\,\epsilon_{\gamma\beta}$ is invariant since for all $g\in  {\rm SL}_2(\mathbb{R})$ we have
\begin{equation*}
 S_{x\,\alpha\beta}\to g_x{}^yg_\alpha{}^\gamma g_\beta{}^\delta S_{y\,\gamma\delta}=S_{x\,\alpha\beta}.
\end{equation*}
Another $ {\rm SL}_2(\mathbb{R})$-invariant tensor is $\eta_{xy}={\rm Tr}(S_xS_y)$, proportional to the Killing-form on $\mathfrak{sl}_2(\mathbb{R})$. We define $\eta^{xy}$ to be its inverse, that is,  $\eta^{xz}\eta_{zy}=\delta^x_y$. We can now compute, for instance, the components of $v\otimes v$ which transform in the representation ${\bf (3,3,3,3)}$ of $G_0$:
\begin{equation*}
\mathbb{T}_{x_1,x_2,x_3,x_4}(v)=S_{x_1\,\alpha_1\beta_1}S_{x_2\,\alpha_2\beta_2}S_{x_3\,\alpha_3\beta_3}S_{x_4\,\alpha_4\beta_4}
v^{\alpha_1\alpha_2\alpha_3\alpha_4}\,v^{\beta_1\beta_2\beta_3\beta_4}\,.
\end{equation*}
Consider now the quartic tensor
\begin{equation*}
\mathfrak{T}^{(1,2,3)}_{x_1,x_2,x_3;\,y_1, y_2,y_3}(v)=\mathbb{T}_{x_1,x_2,x_3,x_4}(v)\mathbb{T}_{y_1,y_2,y_3,y_4}(v)\eta^{x_4y_4}\,.
\end{equation*}
This tensor is a singlet with respect to the last factor ${\rm SL}_2(\mathbb{R})$ in $G_0$.  We can also construct $\mathfrak{T}^{(1,2,4)}_{x_1,x_2,x_4;\,y_1, y_2,y_4}$,  $\mathfrak{T}^{(1,3,4)}_{x_1,x_3,x_4;\,y_1, y_3,y_4}$, and $\mathfrak{T}^{(2,3,4)}_{x_2,x_3,x_4;\,y_2, y_3,y_4}$ , which are singlets with respect to the third, second, and first ${\rm SL}_2(\mathbb{R})$ in $G_0$, respectively. We can project each couple of indices $(x,\,y)$ of these tensors either on the symmetric traceless or on the antisymmetric representations in the product ${\bf 3}\otimes {\bf 3}$ of the corresponding factor ${\rm SL}_2(\mathbb{R})$. However, if we anti-symmetrize an odd number of couples, then we get a vanishing result. We therefore end up with four irreducible tensors for each $\mathfrak{T}^{(i,j,k)}$, giving a total of 16 tensors. Considering the composite index $u=(x_1,x_2,x_3)$ as a single index running over 27 values, the resulting tensors are represented by  16 symmetric $27\times 27$ covariant matrices, which can be used as tensor classifiers. \exdone
\end{example}

In the previous example we have illustrated instances of tensor classifiers which are quadratic and quartic in the components of $v$. Using these, together with $\alpha$-, $\beta$-, and $\gamma$-labels, a  complete classification of the real nilpotent orbits in our main example (see Example \ref{exa:1} and following) can be achieved; similarly in the example discussed in  \cite{Fre:2011uy,Chemissany:2012nb}. In various cases, the $G_0$-orbit is completely determined by the signatures of the corresponding tensor classifiers. The advantage of working with tensor classifiers is that the identification of an orbit does not require the construction of an $\sl_2$-triple and corresponding  $\beta$- and $\gamma$-labels, which,  from a computational point of view, can be difficult. In our example, however, we did not construct a complete set of tensor classifiers (which would alone be sufficient to identify all the orbits in an intrinsic way), since this would require a different analysis, going way beyond the scope of the present work.

\subsection{Gr\"obner bases}\label{sec:gbconj}

Here we recall our set up from Section \ref{secsetup}. We have a
connected algebraic group $\widetilde{G}_0^c\subseteq \GL_k(\C)$, defined
over $\R$, and an isomorphism of algebraic groups $R^c\colon
\widetilde{G}_0^c\to G_0^c$. Furthermore, $\widetilde{G}_0 =
\widetilde{G}^c_0(\R)$ and $G_0 = R^c(\widetilde{G}^c_0(\R))$.
Consider the polynomial ring $Q = \R[x_{11},\ldots,x_{kk}]$ with $k^2$ indeterminates, and
let $p_1,\ldots,p_s\in Q$ such that $\widetilde{G}_0^c$ consists of
all $g\in \GL_k(\C)$ with  $p_1(g)=\ldots=p_s(g)=0$.
Let $m=\dim \g_1$ and fix a basis of $\g_1$; for $g\in\widetilde{G}^c_0(\R)$ denote by $\rho(g)$ be the $m\times m$ matrix over $Q$ representing $g$ with respect to the chosen  basis of $\g_1$.
Now consider nilpotent  $e,e'\in \g_1$. By writing $e$ and $e'$ as coefficient vectors with respect to
the chosen basis of $\g_1$, we get a set of polynomial equations, with
polynomials in $Q$,  equivalent to $e' = \rho(g)e$ for some $g\in \widetilde{G}^c_0(\R)$. To this set we
add $p_1,\ldots,p_s$, and denote the obtained set of polynomials by $P$.
Then $e$ and $e'$ are $G_0$-conjugate if and only if the set of polynomial equations
given by $P$ has a solution over $\R$. In general, this is not easy at all
to decide. However, in practice, computing a Gr\"obner basis of the ideal
generated by $P$, especially when done with respect to a lexicographical
order, often helps enormously. (We refer to \cite{clo} for an introduction
in the subject of Gr\"obner bases.) Indeed, if there is no solution, then
this is often witnessed by the appearance of polynomials in the Gr\"obner
basis that clearly have no simultaneous solution over $\R$, for instance, consider $x_{11}^2+x_{22}^2$ and $x_{11}x_{22}-1$. On the other hand,
if $e$ and $e'$  are conjugate, then the triangular structure of a Gr\"obner
basis with respect to a lexicographical order often helps to find an
explicit conjugating element.

\section{Classification results}\label{sectab}
\noindent
We have applied the methods described in Sections \ref{triples} and \ref{carrier1} to compute the nilpotent $G_0$-orbits in $\g_1$, where $G_0\cong (\SL_2(\R))^4$ and  $\g_1\cong V_2\otimes V_2\otimes V_2\otimes V_2$ with  $V_2$ the natural 2-dimensional $\SL_2(\R)$-module. The results of our computations are listed in Table \ref{tabletot}. The first column of this table lists a label (number) identifying
the orbit under $(\SL_2(\C))^4$ of the given element;  orbits with the same
label lie in the same complex orbit.
The second column lists the
dimension of the orbit, and the third column contains  the representatives that we found; for an explanation of the notation for these representatives we refer to Section \ref{sec:carexa}. The coefficients of our representatives are
$1$ or $-1$, and whenever a sign choice ``$\pm$'' is possible, different choices result in
representatives of different orbits. The final column displays the $\alpha$-, $\beta$-, $\gamma$-, and $\delta$-labels of the orbits; for an explanation of these we refer to Section \ref{SNO} along with Table~\ref{tablelable}.
We use an abbreviation to denote a set of labels, for example, $\{\beta^{(5;1)},\ldots,\beta^{(5;4)}\}$ and  $\{\delta^{(1)},\delta^{(2)}\}$ are abbreviated as $\beta^{(5;1,\ldots,4)}$ and $\delta^{(1,2)}$, respectively; analogous for the other labels.

\begin{table}[htb]
\caption{Real nilpotent $G_0$-orbits}\label{tabletot}
\begin{center}
\resizebox{0.90\columnwidth}{!}{%
\begin{tabular}{c|c|c|c}
{\bf complex orbit} & $\pmb{\dim}$ & \multicolumn{1}{|c|}{{\bf representative}} & {\bf labels} \\
\hline
1 & 5 & $(++--)$ & $\alpha^{(1)}\gamma^{(1;1)}\beta^{(1;1)}$\\
2 & 6 & $(++--)\pm (----)$ &
$\alpha^{(2)}\gamma^{(2;1)}\beta^{(2;1,2)}$\\
3 & 6 & $(++--)\pm (+--+)$ &
$\alpha^{(3)}\gamma^{(3;1)}\beta^{(3;1,2)}$\\
4 & 6 & $(++--)\pm (+-+-)$ &
$\alpha^{(4)}\gamma^{(4;1)}\beta^{(4;1,2)}$\\
5 & 6 & $(++-+)\pm (-+--)$ &
$\alpha^{(4)}\gamma^{(4;2)}\beta^{(4;1,2)}$\\
6 & 6 & $(+++-)\pm (-+--)$ &
$\alpha^{(3)}\gamma^{(3;2)}\beta^{(3;1,2)}$\\
7 & 6 & $(+++-)\pm (++-+)$ &
$\alpha^{(2)}\gamma^{(2;1)}\beta^{(2;1,2)}$\\
8 & 8 & $(++-+)\pm (-+--)\pm (+---)$ & $\alpha^{(5)}\gamma^{(5;4)}\beta^{(5;1,...,4)}$\\
9 & 8 & $(+++-)\pm (-+--)\pm (+---)$ & $\alpha^{(5)}\gamma^{(5;3)}\beta^{(5;1,...,4)}$\\
10 & 8 & $(+++-)\pm (++-+)\pm (+---)$ & $\alpha^{(5)}\gamma^{(5;1)}\beta^{(5;1,...,4)}$\\
11 & 8 & $(+++-)\pm (++-+)\pm (-+--)$ & $\alpha^{(5)}\gamma^{(5;2)}\beta^{(5;1,...,4)}$\\
12 & 9 & $(-+-+)+(++--)+(+--+)+(----)$ & $\alpha^{(6)}\gamma^{(6;4)}\beta^{(6;5)}$\\
12 & & $(-+-+)+(++--)+(+--+)-(----)$ & $\alpha^{(6)}\gamma^{(6;4)}\beta^{(6;3)}$\\
12 & & $(-+-+)+(++--)-(+--+)+(----)$ & $\alpha^{(6)}\gamma^{(6;4)}\beta^{(6;1)}$\\
12 & & $(-+-+)-(++--)+(+--+)+(----)$ & $\alpha^{(6)}\gamma^{(6;4)}\beta^{(6;4)}$\\
12 & & $-(-+-+)+(++--)+(+--+)+(----)$ & $\alpha^{(6)}\gamma^{(6;4)}\beta^{(6;2)}$\\
13 & 9 & $(-++-)+(++--)+(+-+-)+(----)$ & $\alpha^{(6)}\gamma^{(6;3)}\beta^{(6;5)}$\\
13 & & $(-++-)+(++--)+(+-+-)-(----)$ & $\alpha^{(6)}\gamma^{(6;3)}\beta^{(6;4)}$\\
13 & & $(-++-)+(++--)-(+-+-)+(----)$ & $\alpha^{(6)}\gamma^{(6;3)}\beta^{(6;1)}$\\
13 & & $(-++-)-(++--)+(+-+-)+(----)$ & $\alpha^{(6)}\gamma^{(6;3)}\beta^{(6;3)}$\\
13 & & $-(-++-)+(++--)+(+-+-)+(----)$ & $\alpha^{(6)}\gamma^{(6;3)}\beta^{(6;2)}$\\
14 & 9 & $(++++)+(++--)+(+--+)+(+-+-)$ & $\alpha^{(6)}\gamma^{(6;1)}\beta^{(6;5)}$\\
14 & & $(++++)+(++--)+(+--+)-(+-+-)$ & $\alpha^{(6)}\gamma^{(6;1)}\beta^{(6;3)}$\\
14 & & $(++++)+(++--)-(+--+)+(+-+-)$ & $\alpha^{(6)}\gamma^{(6;1)}\beta^{(6;4)}$\\
14 & & $(++++)-(++--)+(+--+)+(+-+-)$ & $\alpha^{(6)}\gamma^{(6;1)}\beta^{(6;1)}$\\
14 & & $-(++++)+(++--)+(+--+)+(+-+-)$ & $\alpha^{(6)}\gamma^{(6;1)}\beta^{(6;2)}$\\
15 & 9 & $(-+++)+(+++-)+(++-+)+(-+--)$ & $\alpha^{(6)}\gamma^{(6;2)}\beta^{(6;5)}$\\
15 & & $(-+++)+(+++-)+(++-+)-(-+--)$ & $\alpha^{(6)}\gamma^{(6;2)}\beta^{(6;1)}$\\
15 & & $(-+++)+(+++-)-(++-+)+(-+--)$ & $\alpha^{(6)}\gamma^{(6;2)}\beta^{(6;3)}$\\
15 & & $(-+++)-(+++-)+(++-+)+(-+--)$ & $\alpha^{(6)}\gamma^{(6;2)}\beta^{(6;4)}$\\
15 & & $-(-+++)+-(+++-)+(++-+)+(-+--)$ & $\alpha^{(6)}\gamma^{(6;2)}\beta^{(6;2)}$\\
16  & 9 & $(+++-)\pm (++-+)\pm (-+--)\pm (+---)$ &
$\alpha^{(6)}\gamma^{(5)}\beta^{(1,...,4)}$ \\
 &  &  &   $\alpha^{(6)}\gamma^{(5)}\beta^{(5)}\delta^{(1,...,4)}$\\
17 & 10 & $(+++-)\pm (++-+)\pm (----)$ & $\alpha^{(7)}\gamma^{(7;1)}\beta^{(7;1,2)}\delta^{(1,2)}$\\
18 & 10 & $(+++-)\pm (-+--)\pm (+--+)$ & $\alpha^{(8)}\gamma^{(8;2)}\beta^{(8;1,2)}\delta^{(1,2)}$\\
19 & 10 & $(++-+)\pm (-+--)\pm (+-+-)$ & $\alpha^{(9)}\gamma^{(9;2)}\beta^{(9;1,2)}\delta^{(1,2)}$\\
20 & 10 & $(-+-+)\pm (+++-)\pm (+---)$ & $\alpha^{(9)}\gamma^{(9;1)}\beta^{(9;1,2)}\delta^{(1,2)}$\\
21 & 10 & $(-++-)\pm (++-+)\pm (+---)$ & $\alpha^{(8)}\gamma^{(8;1)}\beta^{(8;1,2)}\delta^{(1,2)}$\\
22 & 10 & $(++++)\pm (-+--)\pm (+---)$ & $\alpha^{(7)}\gamma^{(7;2)}\beta^{(7;1,2)}\delta^{(1,2)}$\\
23 & 11 & $(+++-)\pm (-+--)\pm (----)\pm (+--+)$ & $\alpha^{(10)}\gamma^{(10;2)}\beta^{(10;1,\dots,4)}\delta^{(1,2)}$\\
24 & 11 & $(-++-)\pm (++-+)\pm (+---)\pm (----)$ &  $\alpha^{(10)}\gamma^{(10;1)}\beta^{(10;1,...,4)}\delta^{(1,2)}$\\
25 & 11 & $(++++)\pm (-+--)\pm (+---)\pm (+--+)$ &  $\alpha^{(10)}\gamma^{(10;3)}\beta^{(10;1,...,4)}\delta^{(1,2)}$\\
26 & 11 & $(++++)\pm (-+--)\pm (+---)\pm (+-+-)$ &  $\alpha^{(10)}\gamma^{(10;4)}\beta^{(10;1,...,4)}\delta^{(1,2)}$\\
27 & 12 & $(-+-+)\pm (+++-)\pm (----)\pm (+--+)$ &  $\alpha^{(11)}\gamma^{(11;4)}\beta^{(11;1,...,4)}\delta^{(1,2)}$\\
28 & 12 & $(-++-)\pm (++-+)\pm (----)\pm (+-+-)$ &  $\alpha^{(11)}\gamma^{(11;3)}\beta^{(11;1,...,4)}\delta^{(1,2)}$\\
29 & 12 & $(++++)\pm (-+--)\pm (+--+)\pm (+-+-)$ &  $\alpha^{(11)}\gamma^{(11;1)}\beta^{(11;1,...,4)}\delta^{(1,2)}$\\
30 & 12 & $(++++)\pm (-++-)\pm (-+-+)\pm (+---)$ &  $\alpha^{(11)}\gamma^{(11;2)}\beta^{(11;1,...,4)}\delta^{(1,2)}$
\end{tabular}}
\end{center}
\end{table}

\begin{table}[htb]
\caption{The list of $\alpha$-$\gamma$-$\beta$-$\delta$-labels}\label{tablelable}
\begin{center}\resizebox{0.70\columnwidth}{!}{%
\begin{tabular}{c|c|c}
{\bf $\pmb{\alpha}$-label} &{\bf$\pmb{\beta}$-$\pmb{\gamma}$-labels} & {\bf$\pmb{\delta}$-label (if present)} \\
\hline
$\alpha^{(1)}=(0,\,1,\,0,\,0)$&$\gamma^{(1;1)}=\beta^{(1;1)}=(1,\,1,\,1,\,1)$&\\
\hline
$\alpha^{(2)}=(2,\,0,\,0,\,0)$&$\gamma^{(2;1)}=\beta^{(2;1)}=(2,\,2,\,0,\,0)$&\\
&$\gamma^{(2;2)}=\beta^{(2;2)}=(0,\,0,\,2,\,2)$&\\
\hline
$\alpha^{(3)}=(0,\,0,\,2,\,0)$&$\gamma^{(3;1)}=\beta^{(3;1)}=(2,\,0,\,2,\,0)$&\\
&$\gamma^{(3;2)}=\beta^{(3;2)}=(0,\,2,\,0,\,2)$&\\
\hline
$\alpha^{(4)}=(0,\,0,\,0,\,2)$&$\gamma^{(4;1)}=\beta^{(4;1)}=(2,\,0,\,0,\,2)$&\\
&$\gamma^{(4;2)}=\beta^{(4;2)}=(0,\,2,\,2,\,0)$&\\
\hline
$\alpha^{(5)}=(1,\,0,\,1,\,1)$&$\gamma^{(5;1)}=\beta^{(5;1)}=(3,\,1,\,1,\,1)$&\\
&$\gamma^{(5;2)}=\beta^{(5;2)}=(1,\,3,\,1,\,1)$&\\
&$\gamma^{(5;3)}=\beta^{(5;3)}=(1,\,1,\,1,\,3)$&\\
&$\gamma^{(5;4)}=\beta^{(5;4)}=(1,\,1,\,3,\,1)$&\\
\hline
$\alpha^{(6)}=(0,\,2,\,0,\,0)$&$\gamma^{(6;1)}=\beta^{(6;1)}=(4,\,0,\,0,\,0)$&\\
&$\gamma^{(6;2)}=\beta^{(6;2)}=(0,\,4,\,0,\,0)$&\\
&$\gamma^{(6;3)}=\beta^{(6;3)}=(0,\,0,\,0,\,4)$&\\
&$\gamma^{(6;4)}=\beta^{(6;4)}=(0,\,0,\,4,\,0)$&\\
&$\gamma^{(6;5)}=\beta^{(6;5)}=(2,\,2,\,2,\,2)$&$\delta^{(1,\dots,4)}$\\
\hline
$\alpha^{(7)}=(2,\,2,\,0,\,0)$&$\gamma^{(7;1)}=\beta^{(7;1)}=(2,\,2,\,4,\,4)$&$\delta^{(1,2)}$\\
&$\gamma^{(7;2)}=\beta^{(7;2)}=(4,\,4,\,2,\,2)$&$\delta^{(1,2)}$\\
\hline
$\alpha^{(8)}=(0,\,2,\,2,\,0)$&$\gamma^{(8;1)}=\beta^{(8;1)}=(2,\,4,\,2,\,4)$&$\delta^{(1,2)}$\\
&$\gamma^{(8;2)}=\beta^{(8;2)}=(4,\,2,\,4,\,2)$&$\delta^{(1,2)}$\\
\hline
$\alpha^{(9)}=(0,\,2,\,0,\,2)$&$\gamma^{(9;1)}=\beta^{(9;1)}=(2,\,4,\,4,\,2)$&$\delta^{(1,2)}$\\
&$\gamma^{(9;2)}=\beta^{(9;2)}=(4,\,2,\,2,\,4)$&$\delta^{(1,2)}$\\
\hline
$\alpha^{(10)}=(2,\,0,\,2,\,2)$&$\gamma^{(10;1)}=\beta^{(10;1)}=(0,\,4,\,4,\,4)$&$\delta^{(1,2)}$\\
&$\gamma^{(10;2)}=\beta^{(10;2)}=(4,\,0,\,4,\,4)$&$\delta^{(1,2)}$\\
&$\gamma^{(10;3)}=\beta^{(10;3)}=(4,\,4,\,4,\,0)$&$\delta^{(1,2)}$\\
&$\gamma^{(10;4)}=\beta^{(10;4)}=(4,\,4,\,0,\,4)$&$\delta^{(1,2)}$\\
\hline
$\alpha^{(11)}=(2,\,2,\,2,\,2)$&$\gamma^{(11;1)}=\beta^{(11;1)}=(8,\,4,\,4,\,4)$&$\delta^{(1,2)}$\\
&$\gamma^{(11;2)}=\beta^{(11;2)}=(4,\,8,\,4,\,4)$&$\delta^{(1,2)}$\\
&$\gamma^{(11;3)}=\beta^{(11;3)}=(4,\,4,\,4,\,8)$&$\delta^{(1,2)}$\\
&$\gamma^{(11;4)}=\beta^{(11;4)}=(4,\,4,\,8,\,4)$&$\delta^{(1,2)}$
\end{tabular}}
\end{center}
\end{table}


\end{document}